\documentclass[11pt,reqno]{amsart}

\usepackage{todonotes}
\usepackage{amsmath}
\usepackage{amsthm}
\usepackage{amssymb}
\usepackage{amsfonts}
\usepackage{amsxtra}
\usepackage{fullpage}
\usepackage{amscd}
\usepackage{mathrsfs}
\usepackage{mathscinet}
\usepackage[all]{xypic}
\usepackage{enumerate}
\usepackage{graphicx}
\usepackage[mathscr]{eucal}
\usepackage{tikz}
\usepackage{tikz-cd}
\usepackage{verbatim}

\let\nc\newcommand
\let\renc\renewcommand

\theoremstyle{plain}
\newtheorem{thm}{Theorem}
\newtheorem{prop}[thm]{Proposition}
\newtheorem{cor}[thm]{Corollary}
\newtheorem{lem}[thm]{Lemma}

\theoremstyle{definition}
\newtheorem{defn}[thm]{Definition}
\newtheorem{example}{Example}
\newtheorem{remark}[thm]{Remark}

\numberwithin{thm}{section}
\nc{\bdm}{\begin{displaymath}}
\nc{\edm}{\end{displaymath}}
\nc{\bthm}{\begin{thm}}
\nc{\ethm}{\end{thm}}
\nc{\blem}{\begin{lem}}
\nc{\elem}{\end{lem}}
\nc{\bcor}{\begin{cor}}
\nc{\ecor}{\end{cor}}
\nc{\bprop}{\begin{prop}}
\nc{\eprop}{\end{prop}}
\nc{\bdef}{\begin{defn}}
\nc{\eddef}{\end{defn}}

\makeatletter
\renewcommand{\subsection}{\@startsection{subsection}{2}{0pt}{-3ex
plus -1ex minus -0.2ex}{-2mm plus -0pt minus
-2pt}{\normalfont\bfseries}} \makeatother

\newcommand{\Lmod}[1]{#1\text{-}{\mathsf{mod}}}

\newcommand{\idot}{{\:\raisebox{2pt}{\text{\circle*{1.5}}}}}

\DeclareMathOperator{\Ext}{\mathrm{Ext}}

\DeclareMathOperator{\End}{\mathrm{End}}

\DeclareMathOperator{\gr}{\mathrm{gr}}
\DeclareMathOperator{\Hilb}{{\mathrm{Hilb}}}

\DeclareMathOperator{\Rep}{\mathrm{Rep}}

\DeclareMathOperator{\Ad}{\mathrm{Ad}}

\DeclareMathOperator{\codim}{\mathrm{codim}}

\newcommand{\beq}{\begin{equation}\label}
\newcommand{\eeq}{\end{equation}}

\DeclareMathOperator{\Spec}{\mathrm{Spec}}

\newcommand{\iso}{{\;\stackrel{_\sim}{\to}\;}}

\DeclareMathOperator{\Hom}{\mathrm{Hom}}

\nc{\Z}{\mathbb{Z}}

\newcommand{\Q}{\mathbb{Q}}

\newcommand{\C}{\mathbb{C}}

\newcommand{\Fun}{\mbox{\mathrm{Fun}}\,}

\nc{\rank}{\textrm{rank} \,}
\nc{\ds}{\dots}

\let\mc\mathcal
\let\mf\mathfrak
\let\mr\mathrm

\nc{\mbf}{\mathbf}
\nc{\LK}{\textsf{Irr}(K)}
\nc{\LW}{\textsf{Irr}(W)}
\nc{\Res}{\mathsf{Res} \, }
\nc{\Ind}{\mathsf{Ind} \, }

\nc{\cont}{\textrm{cont}}

\nc{\msf}{\mathsf}
\newcommand{\git}{\ensuremath{/\!\!/\!}}

\nc{\minusone}{-1}
\nc{\minustwo}{-2}
\nc{\Mod}{\mathrm{Mod} \,}
\nc{\ms}{\mathscr}
\nc{\Frac}{\mathrm{Frac} \,}
\nc{\ra}{\rightarrow}
\nc{\hra}{\hookrightarrow}
\nc{\lab}{\label}
\renc{\O}{\mc{O}}
\nc{\Tan}{\mc{T}}
\nc{\ul}{\underline}
\nc{\s}{\mathfrak{S}}
\nc{\g}{\mf{g}}
\nc{\pg}{\mf{pg}}
\nc{\pa}{\partial}
\nc{\tit}{\textit}
\nc{\Maxspec}{\mathrm{Maxspec} \, }
\nc{\gldim}{\mathrm{gl.dim}}
\nc{\rkm}{\mathrm{rk} \, (\mf{m})}
\nc{\sm}{\mathrm{sm}}
\nc{\PD}{\mathbb{PD}}
\nc{\hilb}{\textrm{Hilb}}
\nc{\<}{\langle}
\renc{\>}{\rangle}
\nc{\T}{\mathbb{T}}
\nc{\X}{\mathfrak{X}}
\nc{\W}{\mathscr{W}}
\nc{\kt}{\mbf{k}}
\nc{\ko}{\mbf{k}(0)}
\nc{\Ok}{\mc{O}_G \boxtimes \kt_X}
\nc{\Oko}{\mc{O}_G \boxtimes \ko_X}
\nc{\OYk}{\mc{O}_Y \boxtimes \kt_X}
\nc{\id}{\msf{id}}
\nc{\A}{\mathbb{A}}
\nc{\Supp}{\mathrm{Supp}}
\nc{\Grat}{\mc{Grat}}
\nc{\Squo}[1]{\A^{(#1)}}
\nc{\twist}{\mathrm{twist}}
\nc{\Cd}{\mc{C}}
\nc{\Span}{\mathrm{Span}}
\nc{\Grass}{\mathrm{Gr}}
\nc{\Fr}{\mathrm{Fr}}
\nc{\pco}[1]{k[V]^{p\mathrm{co} #1}}
\nc{\Irr}{\mathsf{Irr} }
\renc{\o}{\otimes}
\renc{\gr}{\mathsf{gr}}
\nc{\U}{\mathsf{U}}
\nc{\algD}{\mf{D}}

\nc{\hr}{\mf{h}_{\textrm{reg}}}
\nc{\D}{\mathscr{D}}
\nc{\PIdeg}{\mathrm{P.I.-degree}}
\nc{\ch}{\mathrm{ch}}
\nc{\Fl}{\mathcal{F}\mathrm{l}}
\nc{\Stab}{\mathrm{Stab}}
\nc{\Der}{\mathrm{Der}}
\nc{\rightsim}{\stackrel{\sim}{\longrightarrow}}
\nc{\HZ}{H_{\mbf{h},\Z}(\Z_m)}
\nc{\sing}{\mathrm{sing}}
\nc{\dd}{\mathscr{D}}
\nc{\GKdim}{\mathrm{G.K. dim}}
\nc{\PIdegree}{\mathrm{P.I. degree}}
\nc{\Du}{\mathbb{D}}
\renc{\Fun}{\mathrm{Fun}}
\nc{\Xo}{\mathbb{X}}
\nc{\Cs}{\C^{\times}}
\nc{\dQ}{\overline{Q}}
\nc{\bM}{\mathbf{M}}
\nc{\bv}{\mathbf{v}}
\nc{\bw}{\mathbf{w}}
\nc{\Qv}{\mathfrak{M}}
\nc{\barOm}{\overline{\Omega}}
\nc{\Om}{\Omega}
\nc{\Lag}{\mathfrak{L}}
\nc{\free}{\mathrm{free}}
\nc{\Fix}{\mathfrak{F}}
\nc{\IC}{\mathrm{IC}}
\nc{\Ham}{\mathrm{Ham}}
\nc{\bnu}{\boldsymbol{\nu}}
\nc{\tX}{\mc{X}}
\nc{\tY}{\mc{Y}}
\nc{\reg}{\mathrm{reg}}
\nc{\Char}{\mathfrak{X}}
\nc{\Qo}{Q_0}
\nc{\Qn}{Q_1}
\nc{\ba}{\boldsymbol{\alpha}}
\nc{\Nak}[3]{\mf{M}_{{#1}} ({#2},{#3}) }
\nc{\Nakq}[4]{\mf{M}^{{#1}}_{{#2}} ({#3},{#4}) }
\renc{\a}{\alpha}
\nc{\hcf}[1]{\mathsf{hcf}({#1})}
\nc{\stable}{\mathrm{stable}}

\newcommand{\CharG}{\mathscr{X}}
\newcommand{\CharS}{\mathscr{Y}}
\newcommand{\GChar}[2]{\mathscr{Y}(#1,#2)}
\newcommand{\GCharg}[1]{\GChar{#1}{g}}
\nc{\PGL}{{PGL}}
\nc{\PG}{{PG}}
\nc{\G}{{G}}
\nc{\GL}{GL}
\nc{\SL}{SL}
\nc{\Sp}{{Sp}}

\begin{document}

\title{Symplectic resolutions of character varieties}

\begin{abstract}
  In this article we consider the connected component of the identity
  of $\G$-character varieties of 
  compact Riemann surfaces of genus $g > 0$, for connected complex reductive groups $\G$ of type $A$ (e.g., $\SL_n$ and $\GL_n$). 
  We show that these varieties are symplectic singularities and classify which admit symplectic resolutions. The classification reduces to the semi-simple case,
  where
  we show that a resolution exists if and only if either $g=1$ and $G$ is a product of special linear groups of any rank and copies of the group $\PGL_2$, or if $g=2$ and $G = (\SL_2)^m$ for some $m$.
\end{abstract}

\author{Gwyn Bellamy}

\address{School of Mathematics and Statistics, University of Glasgow, University Place, Glasgow, G12 8QQ, UK}

\email{gwyn.bellamy@glasgow.ac.uk}

\author{Travis Schedler}

\address{Department of Mathematics, Imperial College London, South Kensington Campus, London, SW7 2AZ, UK}


\email{trasched@gmail.com}

\subjclass[2010]{16G20;17B63,14D25,58F05,16S80}

\keywords{symplectic resolution, character variety, Poisson variety}

\maketitle

\section{Introduction}

The character varieties associated to the fundamental group of a
topological space have long been objects of study for topologists,
group theorists and algebraic geometers. The character varieties of
reductive groups
associated to the fundamental group of a Riemann surface play a
particularly prominent role in this theory since it has been shown by
Goldman that their smooth locus caries a natural symplectic structure
(which is complex algebraic, and in particular holomorphic). The aim
of this article is to study how this symplectic structure degenerates
along the singular locus of the character variety. This is motivated
by earlier work \cite{BS-srqv} of the authors, where the case of
quiver varieties is considered. Our results for character varieties
are based, in large part, on the ideas developed in \cite{BS-srqv} but
the proofs given here are independent of the latter.

Before explaining the main results of the article, we introduce some
notation. Let $\Sigma$ be a compact Riemann surface of genus $g > 0$
and $\pi$ its fundamental group. Let $\G$ be a connected complex
reductive group whose simple quotients all have type $A$ (i.e., are
all isomorphic to $\PGL_{n_i}$ for some $n_i$). The $\G$-character variety of $\Sigma$ is the affine quotient $\Hom(\pi,\G) \git \, \G$. In general, by \cite{LiCharacter}, $\Hom(\pi,\G)$ is not connected. We will only consider the connected component of the identity, $\Hom(\pi,\G)^{\circ}$,
and set 
$$
\GCharg{\G} := \Hom(\pi,\G)^\circ \git \, \G.
$$
We will refer to $\GCharg{\G}$ as the \textit{connected $\G$-character variety}. Note that, in the case that $\G=\GL_n$ or $\SL_n$, we have $\Hom(\pi,\G)=\Hom(\pi,\G)^\circ$.
 We do not consider the case where $\Sigma$ has punctures; in this case it is natural to impose conditions on the monodromy about the punctures and this situation is addressed in \cite{ST-srmqvcvps}.  

It is shown in \cite{GoldmanInvariantFunctions} that the symplectic structure on the smooth locus extends to a Poisson structure on the whole complex algebraic variety. An important (in both symplectic algebraic geometry and geometric representation theory) class of Poisson varieties with generically non-degenerate Poisson structure are those with symplectic singularities in the sense of Beauville \cite{Beauville}. To have symplectic singularities requires a strong compatibility between the symplectic form and resolutions of singularities. However, once one knows that a given space has symplectic singularities, it allows one to make strong statements about the quantizations and Poisson deformations of the space. For instance, Namikawa \cite{Namikawa} tells us that the (formal) Poisson deformations of such a space are unobstructed, if it is an affine variety. Our first result shows that: 

\begin{thm}\label{thm:slsymp}
  The irreducible variety $\GCharg{\G}$ has symplectic
  singularities.
      \end{thm}
      Since, by  \cite[Proposition 1.3]{Beauville}, all symplectic singularities are rational Gorenstein (meaning Gorenstein with rational singularities, and not meaning the varieties to be rational),
    we deduce:

\begin{cor}\label{cor:rationalGorenstein}
  The variety $\GCharg{\G}$ has rational Gorenstein singularities.
\end{cor} 

The fact that the character variety has rational singularities was
established for large genus in
\cite{AizenbudAvni}. Corollary~\ref{cor:rationalGorenstein} shows that
the bound on the genus $g$ appearing in \cite[Theorem
IX]{AizenbudAvni} is not required, at least when $\G$ is of type
$A$. In the papers \cite{BudurRational,BudurZordan} it is shown that
$\Hom(\pi,\GL_n)$ and $\Hom(\pi,\SL_n)$ also have rational
singularities.

Theorem~\ref{thm:slsymp} immediately raises the question of whether
the character varieties admit (complex algebraic) symplectic
resolutions; these are resolutions of singularities where the
symplectic form on the smooth locus of the singularity extends to a
symplectic form on the whole of the resolution. This is a very strong
condition, and symplectic resolutions are correspondingly rare. The
main result of this article is a complete classification of when these
character varieties admit symplectic resolution. As usual in these
situations, much of the effort is spent in proving that certain
obstructions to the existence of symplectic resolutions exist in many
examples.  Such an obstruction is known to exist when the variety is
singular, $\mathbb{Q}$-factorial, and terminal.

Using Dr\'ezet's Theorem \cite{Drezet} on (local) factoriality of GIT quotients, we show that:    
\begin{thm}\label{thm:slsqfactor}
 The variety $\GCharg{\G}$ is $\mathbb{Q}$-factorial. In the case of $g > 1$ and $(n,g) \neq (2,2)$ it is moreover locally factorial for $\G=\GL_n$ or $\SL_n$.
\end{thm}
Here ``locally factorial'' means that the local rings are unique factorization domains, which is equivalent to saying that every Weil divisor is Cartier; $\mathbb{Q}$-factorial is the weaker condition that some multiple of every Weil divisor is Cartier.  

To apply this to the question of existence of symplectic resolutions, we need to estimate the codimension of the singular locus. Thanks to   Theorem \ref{thm:slsymp} together with
\cite[Theorem 2.3]{Kaledinsympsingularities},
  the singular locus of $\GCharg{\G}$ is a finite union of even-dimensional strata: the symplectic leaves. In the case of $\G=\GL_n$, this is proved in \cite{Sim-mrfgspv2} and is in \cite{SikoraAbelian} when $g = 1$; by our methods, we can also reduce the general case to these ones.
Note that, if $\G$ is abelian, i.e., a torus, then $\GCharg{\G} \cong \G^{2g}$, which is not interesting for our purposes.
\begin{prop}\label{p:codim}
  Assume that $\G$ is nonabelian.  Then $\GCharg{\G}$ is
  singular. Its singular locus has codimension at least four if and
  only if $g > 1$ and, in the case $g=2$, no simple quotient of $\G$
   is isomorphic to $\PGL_2$.
  \end{prop}


  Thanks to \cite{NamikawaNote}, when a symplectic singularity has singular locus of codimension at least four, it has terminal singularities. On the other hand, it is well known that a singular $\mathbb{Q}$-factorial terminal variety does not admit a proper crepant resolution, and hence not a symplectic one: see, e.g., the proof of \cite[Theorem 6.13]{BS-srqv}.  Thus,
  Theorem \ref{thm:slsqfactor} and Proposition \ref{p:codim} immediately imply:
\begin{cor}\label{cor:slres}
  Assume $\G$ is nonabelian, $g > 1$ and, if $g=2$, then no simple quotient of $\G$ is isomorphic to $\PGL_2$.
  Then $\GCharg{\G}$  does not admit a proper symplectic resolution. The same holds for any singular open subset. 
\end{cor}
In particular, the above implies that, for $G=\SL_n, \GL_n$, or $\PGL_n$, then the character variety $\GCharg{\G}$ does not admit a symplectic resolution when $n,g \geq 2$ and $(n,g) \neq (2,2)$. In these cases we give a description of all the symplectic leaves, and hence the stratification of the singular locus into iterated singular loci (all even-dimensional).

\begin{remark}
  In the case $\G=\GL_n$, parallel to Remark 1.14 of \cite{BS-srqv}, we can give an alternative proof of the first statement of Corollary \ref{cor:slres} using formal localization, reducing to the quiver variety case which is considered in that article. The formal
neighborhood of the identity of $\GCharg{\GL_n}$ is well-known to identify with the formal neighborhood of $(0,\ldots,0)$ in the quotient 
$$
\left\{ (X_1,Y_1, \ds, X_g,Y_g) \in \mathfrak{gl}_n
\ \left| \ \sum_{i = 1}^d [X_i,Y_i] = 0 \right. \right\} \git \, \GL_n.
$$
See Proposition \ref{p:formal-neighborhood} for more details.
The above quotient is
the quiver variety associated to the quiver $Q$ with
one vertex and $g$ arrows. However, we cannot directly conclude Theorem \ref{thm:slsqfactor} using formal
localization, and neither the stronger last statement of Corollary
\ref{cor:slres}.
\end{remark}
\begin{remark}
  As in the discussion after Corollary 1.11 of \cite{BS-srqv}, one can obtain singular open subsets $U \subseteq \GCharg{\GL_n}$ in the case $g > 1$ and $(n,g) \neq (2,2)$ for which the formal neighborhood of every point does admit a resolution, even though the entire $U$ does not admit one by Corollary~\ref{cor:slres}. Indeed, as mentioned in the preceding remark, the formal neighborhoods identify with those of quiver varieties. Hence one example is given analogously to the one in \cite{BS-srqv} (after Corollary 1.11 therein):
  for $n=2$ and $g \geq 3$, take $U$ to be the complement
  of the locus of representations of the form $Y^{\oplus 2}$ for $Y$
  one-dimensional (and hence irreducible).  This locus is singular, terminal, and locally factorial, but it is not formally locally $\Q$-factorial.
\end{remark}

\subsection{Construction of symplectic resolutions}
As the preceding has made clear, in most cases there will not be a symplectic resolution.  Let us explain how to construct the ones that do exist.

First suppose that the genus is one.  For the general linear group,
$\GChar{\GL_n}{1}$ is well known to be isomorphic to the symmetric product
$S^n(\Cs \times \Cs):= (\Cs \times \Cs)^n/S_n$ (see, e.g., \cite{JosephChevalley},\cite{AlmostCommutingVariety}, as explained in Section \ref{sec:genus1} below).
The Hilbert scheme produces a symplectic resolution,
$\Hilb^n(\Cs \times \Cs) \to \GChar{\GL_n}{1}$, by the Hilbert--Chow map. That $\Hilb^n(\Cs \times \Cs)$ is symplectic, owing to the fact that $\Cs \times \Cs$ is a holomorphic symplectic surface, was observed in \cite{HilbertSymplectic}. For $\SL_n < \GL_n$, we have $\GChar{\SL_n}{1} \subset \GChar{\GL_n}{1}$ and the restriction of the Hilbert--Chow morphism to the preimage of $\GChar{\SL_n}{1}$ is a symplectic resolution of the latter. This preimage is the smooth symplectic subvariety $\Hilb^n_0(\Cs \times \Cs) \subseteq \Hilb^n(\Cs \times \Cs)$, of schemes of length $n$ in $\Cs \times \Cs$ such that the product of all pairs in the support is $(1,1)$.  See Section \ref{sec:genus1} for more details on these facts.

Similarly, we can produce a resolution for the group $\PGL_2$: here one obtains that $\GChar{\PGL_2}{1} \cong \GChar{\SL_2}{1}/\Z_2^2$, which allows us to produce a partial resolution via $\Hilb^2_0(\Cs \times \Cs)/\Z_2^2$, and finally blow it up to produce a symplectic resolution.  

Finally, for genus at least two, there is one exceptional case where we have a resolution, namely for $\GChar{\GL_2}{2}$, and hence also for $\GChar{\SL_2}{2}$: in these cases a resolution was constructed in \cite{LSOGrady} (in a slightly different context). Again, in this case it turns out one only needs to blow up the reduced singular locus. Put together, we obtain the following:

\begin{thm} \label{t:constr-res}
  \begin{itemize}
  \item[(i)] For $g=1$, we have the projective symplectic resolutions
    \begin{equation}
      \Hilb^n(\Cs \times \Cs) \to \GChar{\GL_n}{1}, \quad \Hilb^n_0(\Cs \times \Cs) \to \GChar{\SL_n}{1},
    \end{equation}
    given by the Hilbert--Chow morphism, and
    \begin{equation}
      \widetilde{\Hilb^2_0(\Cs \times \Cs)/\Z_2^2} \to \GChar{\PGL_2}{1},
    \end{equation}
    given by blowing up the (reduced) singular locus (type $A_1$ surface singularities only).
  \item[(ii)] For $g=2$, we have the symplectic resolutions
    \begin{equation}
      \widetilde{\GChar{\GL_2}{2}} \to \GChar{\GL_2}{2}, \quad \widetilde{\GChar{\SL_2}{2}} \to \GChar{\SL_2}{2},
    \end{equation}
    obtained by blowing up the reduced singular locus.
  \end{itemize}
\end{thm}
The theorem will be proved in Sections \ref{ss:22} and \ref{sec:genus1} below. In Section \ref{ss:min} below, we explain the closest approximation to such a resolution that exists in general.


\subsection{Classification of symplectic resolutions}\label{ss:class}
Many of the preceding results (in particular Theorems \ref{thm:slsymp} and \ref{thm:slsqfactor}) reduce easily to the case of general and special linear groups, handled in Section \ref{s:gl-sl} below.
We
now explain how to strengthen these results to a full classification
of character varieties $\GCharg{\G}$ admitting symplectic
resolutions, for $\G$ a (necessarily connected) reductive group whose simple quotients are all of type $A$.
The first step is to reduce to the case where $\G$ is semi-simple. Here ``semi-simple'' means reductive with finite center.

Note that a general group $\G$ as above has the form:
$$
1 \to Z \to H \times K \to \G \to 1, 
$$
where $H$ is a torus, $K$ is a semi-simple connected group and $Z < H \times Z(K)$ finite. Moreover, $K$ can be chosen so that the the composition $Z \to H \times K \to H$ is
injective.
Since $H$
is central, this implies that $Z^{2g}$ acts freely on
$\GChar{H \times K}{g} = \GChar{H}{g} \times \GChar{K}{g}$. We have $\GCharg{\G} \cong \GCharg{H \times K}/Z^{2g}$ (see Section \ref{ss:fin-quotient} below), and the two
have the same singularities.  So for most of our purposes, we can
replace $\G$ by $K$.

In these terms our main result on symplectic resolutions takes the following general form:
\begin{thm}\label{t:class-res} 
	Assume that $\G$ is nonabelian. The following are equivalent:
  \begin{itemize}
  \item[(a)] The character variety $\GCharg{\G}$ admits a projective symplectic resolution.
  \item[(b)] The character variety $\GCharg{K}$ admits a projective symplectic resolution.
  \item[(c)] One of the following two conditions holds:
    \begin{itemize}
    \item[(i)] $g=1$ and $K \cong \prod_j \SL_{n_j} \times \PGL_2^m$ for some $m \geq 0$;
    \item[(ii)] $g=2$ and $K \cong \SL_2^m$ for some $m \geq 1$.
    \end{itemize}
\end{itemize}
\end{thm}

In the cases in part (c), a resolution can be constructed thanks to Theorem \ref{t:constr-res}. Namely, by taking products, we get a resolution $\widetilde{\GCharg{K}} \to \GCharg{K}$, and then $\bigl(\GCharg{H} \times \widetilde{\GCharg{K}})/Z^{2g}$ produces a resolution of $\GCharg{G}$.

We first prove these results for $\G = \GL_n$ or $\SL_n$, and then in Section~\ref{sec:generaltypeAchar} we deduce the results for arbitrary type $A$ reductive groups. Similar techniques are applicable to Hitchin's moduli spaces of semistable Higgs bundles over smooth projective curves; see \cite{Tir-srHms}.

Our classification of character varieties of type $A$ admitting symplectic resolutions has recently been used in \cite{PWSympres} to prove the $P=W$ conjecture in all cases where a symplectic resolution exists. 

\subsection*{Conventions}

Throughout, a variety will mean a reduced, quasi-projective scheme of finite type over $\C$. If $X$ is a (quasi-projective) variety equipped with the action of a reductive algebraic group $\G$, then $X \git \, \G$ will denote the good quotient (when it exists). In this case, let $\xi: X \rightarrow X \git \, \G$ denote the quotient map. Then each fibre $\xi^{-1}(x)$ contains a unique closed $\G$-orbit. Following Luna, this closed orbit is denoted $T(x)$.  

\subsection*{Acknowledgments}

The first author was partially supported by EPSRC grant EP/N005058/1.
The second author was partially supported by NSF Grant DMS-1406553.
The authors are grateful to the University of Glasgow for the
hospitality provided during the workshop ``Symplectic representation
theory'', where part of this work was done, and the second author to
the 2015 Park City Mathematics Institute as well as to the Max Planck
Institute for Mathematics for excellent working environments.

We would like to thank Victor Ginzburg for suggesting we
consider character varieties. We would also like to thank David
Jordan, Johan Martens and Ben Martin for their many explanations
regarding character varieties.

Finally, we would like to thank the referee for a detailed reading of the paper and several suggestions that have improved the exposition. 

\section{Character varieties for general and special linear groups}\label{s:gl-sl}

Recall from the introduction that $\Sigma$ is a compact Riemannian surface of genus $g > 0$ and $\pi$ is its fundamental group. For a connected reductive group $\G$ we have defined the character varieties
$$
\GCharg{\G} = \Hom(\pi,\G)^{\circ} \git \, \G, 
$$
where $\Hom(\pi,\G)^{\circ} \subset \Hom(\pi,\G)$ is the connected component containing the trivial representation. These are affine varieties. By \cite{LiCharacter}, $\Hom(\pi,\G)^{\circ} = \Hom(\pi,\G)$ for both $\GL_n$ and $\SL_n$.

\subsection{Quasi-Hamiltonian reduction} In this section
we will only consider $\G = \GL_n$ or $\SL_n$. If $g > 1$ then the (complex) dimension of $\GCharg{\SL_n}$ is $2 (g-1)(n^2 - 1)$, and when $g = 1$, it has dimension $2(n-1)$. On the other hand $\dim \GCharg{\GL_n} = \dim \GCharg{\SL_n} +2g$ always. We begin by recalling the basic properties of the affine varieties $\Hom(\pi,\GL)$ and $\GCharg{\GL_n}$:

\begin{thm}\cite[\S 11]{Sim-mrfgspv2}, \cite{Richardsonntuples}.
\label{thm:characterproperties}
	Assume $g > 1$. \label{thm:Etingofsum} 
\begin{enumerate} 
\item Both $\Hom(\pi,\GL_n)$ and $\GCharg{\GL_n}$ are reduced, irreducible, and normal. \label{thm:Etingofsum:irr}
\item $\Hom(\pi,\GL_n)$ is a complete intersection in $\GL_n^{2g}$. 
\item The generic points of $\Hom(\pi,\GL_n)$ and $\GCharg{\GL_n}$ correspond to irreducible representations of the fundamental group $\pi$. 
\end{enumerate}
\end{thm}  

As shown originally by Goldman \cite{GoldmanSymplectic}, the variety $X$ has a natural Poisson structure. This Poisson structure becomes clear in the realization of these spaces as quasi-Hamiltonian reductions; see \cite{Liegroupvaluedmoment}, where it is shown that the symplectic structure defined by Goldman on the smooth locus of $X$ agrees with the Poisson structure of $X$ as a quasi-Hamiltonian reduction. In particular, if $C_{(1,n)}$ denotes the dense open subset of $X$ parameterizing simple representations of $\pi$, then it is shown in \cite{Liegroupvaluedmoment} that the Poisson structure on $C_{(1,n)}$ is non-degenerate. 

It will be useful for us to reinterpret the quasi-Hamiltonian reduction $\GCharg{\GL_n}$ as a moduli space of semi-simple representations of the multiplicative preprojective algebra. Let $Q$ be the quiver with a single vertex and $g$ loops, labeled $a_1, \ds, a_g$. Let $a_i^*$ denote the loop dual to $a_i$ in the doubled quiver $\overline{Q}$. Associated to $Q$ is the multiplicative preprojective algebra $\Lambda(Q)$, as defined in \cite{CBShaw}. Namely, $\C \overline{Q} \rightarrow \Lambda(Q)$ is the universal homomorphism such that each $1 + a_i a^*_i$ and $1 + a_i^* a_i$ is invertible and 
$$
\prod_{i = 1}^g (1 + a_i a_i^*)(1 + a_i^* a_i)^{-1} = 1.
$$
Here the product is ordered. Following \cite{CBMonodromy}, let $\Lambda(Q)'$ denote the universal localization of $\Lambda(Q)$, where each $a_i$ is also required to be invertible. Let $(T^* \Rep(Q,n))^{\circ} $ denote the space of all $n$-dimensional representations $(A_i,A_i^*)$ of $\C \overline{Q}$ such that $1 + A_i A_i^*, 1 + A_i^* A_i$ and $A_i$ are invertible for all $i$. It is an open, $\GL_n$-stable affine subset of $T^* \Rep(Q,n)$. The action of $\GL_n$ on $(T^* \Rep(Q,n))^{\circ}$ is quasi-Hamiltonian, with multiplicative moment map 
$$
\Psi : \Rep(\Lambda(Q)',n) \rightarrow \GL_n, \quad (A_i,A_i^*) \mapsto \prod_{i = 1}^g (1 + A_i A_i^*)(1 + A_i^* A_i)^{-1}.
$$
As noted in Proposition 2 of \cite{CBMonodromy}, the category $\Lmod{\Lambda(Q)'}$ of finite dimensional $\Lambda(Q)'$-modules is equivalent to $\Lmod{\pi}$, in such a way that we have a $\GL$-equivariant identification
$$
\Psi^{-1}(1) \iso \Hom(\pi, \GL_n), \quad (A_i,A_i^*) \mapsto (A_i,B_i) = (A_i, A_i^{-1} + A_i^*). 
$$
Hence, we have an identification of Poisson varieties
$$
\Psi^{-1}(1) \git \, \GL_n = \GCharg{\GL_n}. 
$$
See \cite{DoublePoissonAlgebras} for further details. 

\subsection{Symplectic singularities}\label{ss:ss}

We recall from \cite[Definition 1.1]{Beauville} that a variety $X$ is said to be a symplectic singularity if it is normal, its smooth locus has a symplectic $2$-form $\omega$, and for any resolution of singularities $f \colon Y \rightarrow X$, the rational $2$-form $f^* \omega$ is regular. Moreover, $f$ is said to be a symplectic resolution if the $2$-form $f^* \omega$ is also non-degenerate. In particular, this makes $Y$ an algebraic symplectic manifold. 

The space $\GCharg{\GL_n}$ has a stratification by representation type, which is also the stratification by stabilizer type; see \cite[Theorem 5.4]{MarsdenWeinsteinStratification}. We say that a \textit{weighted partition} $\nu$ of $n$ is a sequence $(\ell_1, \nu_1; \ds ; \ell_k,\nu_k)$, where each $\ell_i$ and $\nu_i$ is a positive integer and
$$
\nu_1 \ge \nu_2 \ge \cdots, \quad \sum_{i = 1}^k \ell_i \nu_i = n.
$$ 

\begin{lem}\label{lem:stratadimChar}
Assume $n,g > 1$.  
\begin{enumerate}
\item The strata $C_{\nu}$ of $\GCharg{\GL_n}$ are labelled by weighted partitions of $n$ such that    
$$
\dim C_{\nu} = 2 \left( k + (g-1) \sum_{i = 1}^k \nu_i^2 \right).
$$ 
\item If $(n,g) \neq (2,2)$, then
  $\dim \GCharg{\GL_n} - \dim C_{\nu} \ge 4$ for all $\nu \neq (1,n)$.
\item If $(n,g) \neq (2,2)$ and $\nu \neq (1,n)$, then $\dim \GCharg{\GL_n}-\dim C_{\nu} \ge 8$ unless either 
\begin{itemize}
	\item[(i)] $(n,g)=(3,2)$ and $\nu=(1,2;1,1)$; or 
	\item[(ii)] $(n,g)=(2,3)$ and $\nu=(1,1;1,1)$.
\end{itemize}
\end{enumerate}
\end{lem}

\begin{proof}
By Theorem \ref{thm:Etingofsum}, the set of points $C_{(1,n)}$ in $\GCharg{\GL_n}$ parameterizing irreducible representations of $\pi$ is a dense open subset contained in the smooth locus. Therefore $\dim C_{(1,n)} = 2(1 + n^2(g-1))$. An arbitrary semi-simple representation of $\pi$ of dimension $n$ has the form $x= x_1^{\oplus \ell_1} \oplus \cdots \oplus x_k^{\oplus \ell_k}$, where the $x_i$ are pairwise non-isomorphic irreducible $\pi$-modules of dimension $\nu_i$ and $n = \sum_{i = 1}^k \ell_i \nu_i$. Thus, the representation type strata correspond to weighted partitions of $n$. Let $C_{\nu}$ denote the locally closed subvariety of all such representations. Though the $x_i$ are pairwise distinct, there will generally exist $i,j$ for which $\nu_i = \nu_j$. Reordering if necessary, we write the multiset $\{ \! \{ \nu_1, \ds, \nu_k \} \! \}$ as $\{ \! \{ m_1 \cdot \nu_1, \ds, m_r \cdot \nu_r \}\! \}$, with $\nu_i \neq \nu_j$ and $r \le k$. Then
$$
C_{\nu} \cong S^{m_1,\circ} C_{(1,\nu_1)} \times \cdots \times S^{m_r,\circ} C_{(1,\nu_r)},
$$
where $S^{n,\circ} X$ is the open subset of $S^n X$ consisting of $n$ pairwise distinct points. Thus, 
$$
\dim C_{\nu} = \sum_{i = 1}^r 2(1 + \nu_i^2(g-1)) m_i = 2 \left( k + (g-1) \sum_{i = 1}^k \nu_i^2 \right). 
$$
For the second part, notice that 
\begin{align}\label{e:dimstt}
\dim \GCharg{\GL_n}- \dim C_{\nu} & = 2 (n^2 (g - 1) +1) -  2 \sum_{i = 1}^k (\nu_i^2 (g-1 ) + 1) \notag \\
& = 2 (g-1) \sum_{i,j = 1}^k (\ell_i \ell_i - \delta_{i,j}) \nu_i \nu_j - 2 (k-1).
\end{align}
Since $ \sum_{i,j = 1}^k (\ell_i \ell_j - \delta_{i,j}) \nu_i \nu_j - (k-1) \ge 1$, we clearly have  $\dim \GCharg{\GL_n} - \dim C_{\nu} \ge 4$ when $g > 2$. When $g = 2$, a simple computation shows that $\dim \GCharg{\GL_n} - \dim C_{\nu} = 2$ if and only if $n = 2$ and $\nu = (1,1;1,1)$. 


For the third part, we use again \eqref{e:dimstt}, noticing the
following points: the RHS of \eqref{e:dimstt} is increasing in $g$;
the RHS is increased if we replace $(\ell_i,n_i)$ by
$(\ell_i-1,n_i);(1,n_i)$; the RHS is increased if we replace $(1,a)$
and $(1,b)$ by $(1,a+b)$ (when $a+b < n$); and for $a > b > 1$, the
RHS is increased if we replace $(1,a)$ and $(1,b)$ by $(1,a+1)$ and
$(1,b-1)$.  Since it suffices to prove the inequality after performing
operations that increase the RHS, the result follows once we observe
that the inequality holds in the following cases: (i) $\nu=(1,n-1;1,1)$ whenever $n \geq 4$ as well as $(1,1;1,1;1,1)$; (ii)
for $\nu=(1,1;1,1)$ whenever $g \geq 4$, as well as $\nu=(2,1)$
for $g=3$.
\end{proof}

Recall that $\GCharg{\SL_n}$ is the character variety associated to the compact Riemann surface $\Sigma$, of genus $g$, with values in $\SL_n$. Let $T \cong (\Cs)^{2g}$ denote the $2g$-torus. Let $\varrho : \Hom(\pi,\GL) \rightarrow T$ be the map sending $(A_i,B_i)$ to $(\det(A_i), \det(B_i))$. This map is $\GL_n$-equivariant, where the action on $T$ is trivial. Moreover, it fits into a commutative diagram of $\GL_n$-varieties
\begin{equation}\label{eq:trivialbundle1}
\xymatrix{
	\Hom(\pi,\SL_n) \times_{\Z_n^{2g}} T \ar[rr]^{\sim} \ar[dr]_{\mathrm{pr}} & & \Hom(\pi,\GL_n) \ar[dl]^{\varrho}\\
	& T & 
}
\end{equation}
where $\Z_{n}^{2g}$ acts freely on $T$, and the map $\Hom(\pi,\SL_n) \times T \rightarrow \Hom(\pi,\GL_n)$ sends $((A_i,B_i),(t_i,s_i))$ to $(t_i A_i, s_i B_i)$. Therefore it descends to a commutative diagram
\begin{equation}\label{eq:trivialbundle2}
\xymatrix{
	\GCharg{\SL_n} \times_{\Z_n^{2g}} T \ar[rr]^{\sim} \ar[dr]_{\mathrm{pr}} & & \GCharg{\GL_n} \ar[dl]^{\varrho}\\
	& T & 
}
\end{equation}
where $\Z_{n}^{2g}$ acts freely on $\GCharg{\SL_n} \times T$. As in \cite[Corollary 2.6]{FLfreegroups} (see also \cite[Proposition 5]{SikoraSO}), we deduce that: 

\begin{lem}\label{lem:charSfiberbundle} 
	The variety $\GCharg{\GL_n}$ is an \'etale locally trivial fiber bundle over $T$ with fiber $\GCharg{\SL_n}$.
\end{lem}

Lemma~\ref{lem:charSfiberbundle} implies that the dimension estimates in Lemma~\ref{lem:stratadimChar}(2),(3) hold for $\GCharg{\SL_n}$ too.

\begin{proof}[Proof of Theorem \ref{thm:slsymp} for $\GCharg{\SL_n}$ and $\GCharg{\GL_n}$]
  When $g =1$ the claim follows from Proposition~\ref{prop:genusone}. The case $(n,g) = (2,2)$ is dealt with in
	Corollary~\ref{cor:22symp} below. The case $n=1$ is trivial.
	
	Therefore we assume $n,g > 1$ and $(n,g) \neq (2,2)$. Recall from Theorem~\ref{thm:characterproperties} that $\GCharg{\GL_n}$ is normal. First, we deduce from this that $\GCharg{\SL_n}$ is also normal. Lemma \ref{lem:charSfiberbundle} implies that $\GCharg{\SL_n}$ is an irreducible variety of dimension $2(g-1)(n^2 - 1)$ since $\dim \GCharg{\GL_n} = 2 n^2(g-1) + 2$. If $\GCharg{\SL_n}$ were not normal, then $\GCharg{\SL_n} \times T$ would also not be normal. But the fact that $\GCharg{\SL_n} \times_{\Z_n^{2g}} T \cong \GCharg{\GL_n}$ is normal, and the map $\GCharg{\SL_n} \times T \rightarrow \GCharg{\GL_n}$ is \'etale, implies by \cite[Proposition 3.17]{Milne} that $\GCharg{\SL_n} \times T$ is normal. Thus, $\GCharg{\SL_n}$ is normal. 
	
	By Theorem \ref{thm:Etingofsum}, the Poisson structure on the dense open subset $C_{(1,n)}$ of $\GCharg{\GL_n}$ is non-degenerate. This implies that the Poisson structure on the whole of the smooth locus is non-degenerate since the complement to $C_{(1,n)}$ in $\GCharg{\GL_n}$ has codimension at least four. The identification $\GCharg{\SL_n} \times_{\Z_n^{2g}} T \cong \GCharg{\GL_n}$ of Lemma \ref{lem:charSfiberbundle} is Poisson, where we equip $\GCharg{\SL_n} \times T$ with the product Poisson structure. We deduce that the Poisson structure on the smooth locus of $\GCharg{\SL_n}$ is non-degenerate, and the singular locus of $\GCharg{\SL_n}$ has codimension at least $4$ when $(n,g) \neq (2,2)$. 
	
	Since $\GCharg{\SL_n}$ and $\GCharg{\GL_n}$ are normal and their singular locus has codimension at least four,  it follows from Flenner's Theorem \cite{Flenner} that they have symplectic singularities.
\end{proof}

In the proof of Theorem \ref{thm:slsymp}, we have only used the fact that $C_{(1,n)}$ is contained in the smooth locus. 
As shown by Goldman \cite{GoldmanSymplectic}, the singular locus 
is precisely the complement to $C_{(1,n)}$.

\begin{lem}\label{lem:singlocusGoldman}
	The singular locus of $\GCharg{\SL_n}$ and $\GCharg{\GL_n}$ is precisely the complement to the open stratum $C_{(1,n)}$.
\end{lem} 

\begin{proof}
	Let $\G = \GL_n$ or $\SL_n$ and $PG := \G / Z(\G)$. The explicit description of the strata given in Lemma~\ref{lem:stratadimChar} shows that if $\rho \in \Hom(\pi,\G)$ is a semi-simple representation that is not simple then its stabilizer under $\PG$ has strictly positive dimension. Therefore, it follows from the formula for the dimension of $T_{[\rho]} X$ given in Section 1.5 of \cite{GoldmanSymplectic} that $[\rho]$ belongs to the singular locus of $X$. 
\end{proof}

As for quiver varieties, the symplectic leaves of the character variety $X$ are the stabilizer type strata $C_{\nu}$. For $g > 1$ and $G = \GL_n$, this result is contained in \cite{Sim-mrfgspv2} and is in \cite{SikoraAbelian} when $g = 1$. The same statement holds for $\GCharg{\SL_n}$ using Lemma~\ref{lem:charSfiberbundle}.  

\subsection{The case $(n,g) = (2,2)$}\label{ss:22}

The case $(n,g) = (2,2)$ can be thought of as a ``local model'' for
the moduli space $M_{2v}$ of semistable shaves with Mukai vector $2v$
on an abelian or $K3$ surface, where $v$ is primitive with $\langle v, v \rangle = 2$. Therefore, we are able to emulate the arguments of Lehn--Sorger \cite{LSOGrady} and Kaledin--Lehn \cite{KaledinLehn} in this case.  Lemma~\ref{lem:stratadimChar}(1) says that $\GChar{\GL_2}{2}$ has three strata,
$C_{(1,2)}$ consisting of simple representations $E$, $C_{(1,1;1,1)}$
consisting of semi-simple representations $E = F_1 \oplus F_2$, where
$F_1$ and $F_2$ are a pair of non-isomorphic one-dimensional
representations of $\pi$, and $C_{(2,1)}$ the stratum of semi-simple
representations $E = F^{\oplus 2}$, where $F$ is a one-dimensional
representation. By Lemma~\ref{lem:singlocusGoldman}, the singular locus of
$\GChar{\GL_2}{2}$ equals
$\overline{ C}_{(1,1;1,1)} = C_{(1,1;1,1)} \sqcup C_{(2,1)}$.

\begin{thm}\label{thm:Blowup22LS}
  The blowup $\sigma : \widetilde{\GChar{\GL_2}{2}} \rightarrow \GChar{\GL_2}{2}$ along the
  reduced ideal defining the singular locus of $\GChar{\GL_2}{2}$ defines a
  semi-small resolution of singularities.
\end{thm}
Here, a map $X \to Y$ is called \emph{semi-small} if $\dim X \times_Y X = \dim X$.
\begin{proof}

Fix points $E \in C_{(1,1;1,1)}$ and $E' \in C_{(2,1)}$. The character variety $\GCharg{\GL_n}$ is the moduli space of representations of the group algebra $\C[\pi]$ 	of the fundamental group $\pi$ of a genus $g$ surface. The group algebra $\C[\pi]$ is 	a two-dimensional Calabi--Yau algebra. Then \cite[Theorem 6.3, Theorem 6.6]{BGV-NHsCY2}\footnote{Thanks to
	Raf Bocklandt for pointing out these results.} and \cite[Corollary 5.21]{Kaplan-Schedler-mpa} imply that the formal neighborhoods of $E$ and $E'$ in $\GCharg{\GL_n}$ can be identified with the formal neighborhood of $0$ in the quiver varieties associated to the corresponding ext-quiver of these points. In the case of $E$, this is the quiver with two vertices, an arrow in each direction and a loop at each vertex; the dimension vector is $(1,1)$. For $E'$ this is the quiver with one vertex and two loops; the dimension vector is $(2)$. 
	
Let $\mc{O} = \{ B \in \mf{sp}(4) \ | \ B^2 = 0, \ \mathrm{rk} B = 2 \}$. The set $\mc{O}$ is a $6$-dimensional nilpotent adjoint $\Sp(4)$-orbit in $\mf{sp}(4)$. We define $\mc{N}$ to be the closure of $\mc{O}$ in $\mf{sp}(4)$; this variety is a union of three nilpotent orbits. The quiver varieties associated to the above ext-quivers are then isomorphic to $\C^8 \times (\C^2 / \Z_2)$ and $\C^4 \times \mc{N}$ respectively; see \cite[Theorem 5.1]{BS-srqv}. By Artin approximation \cite[Corollary~1.6]{AnalyticArtin}, these identification of formal neighborhoods lift to isomorphisms of analytic germs

$$
(\GChar{\GL_2}{2},E) \cong ( \C^8 \times (\C^2 / \Z_2), 0), \quad
(\GChar{\GL_2}{2},E') \cong ( \C^4 \times \mc{N}, 0).
$$
Clearly, blowing up $\C^8 \times (\C^2 / \Z_2)$ along
the singular locus gives a semi-small resolution of singularities. The
key result \cite[Remark 5.4]{KaledinLehn}, see also
\cite[Th\'eor\`eme 2.1]{LSOGrady}, says that blowing up
along the reduced ideal defining the singular locus in
$\C^4 \times \mc{N}$ also produces a semi-small resolution of
singularities.
\end{proof}

\begin{cor}\label{cor:22symp}
  The blowup $\widetilde{\GChar{\GL_2}{2}}$ of $\GChar{\GL_2}{2}$ along the
  reduced ideal defining the singular locus of $\GChar{\GL_2}{2}$ is a
  smooth symplectic variety and $\GChar{\GL_2}{2}$ has symplectic singularities.
\end{cor}

\begin{proof}
  Let $\sigma : \widetilde{\GChar{\GL_2}{2}} \rightarrow \GChar{\GL_2}{2}$ denote the blowup map. The singularities of $\GChar{\GL_2}{2}$ in a an analytic neighborhood of a point in $C_{(1,1;1,1)}$ are equivalent to an $A_1$ singularity. Therefore the pullback $\sigma^* \omega$ of the symplectic $2$-form $\omega$ on the smooth locus of $\GChar{\GL_2}{2}$ extends to a symplectic $2$-form on $\sigma^{-1}(U)$, where $U$ is the open set $C_{(1,2)} \cup C_{(1,1;1,1)}$. Since $\sigma$ is semi-small, $\sigma^{-1}(C_{(2,1)})$ has codimension at least $3$ in $\widetilde{\GChar{\GL_2}{2}}$. Therefore, $\sigma^* \omega$ extends to a symplectic $2$-form on the whole of $\widetilde{\GChar{\GL_2}{2}}$. Since $\GChar{\GL_2}{2}$ is normal and $\sigma$ is birational, it is a symplectic resolution. In particular, $\GChar{\GL_2}{2}$ has symplectic singularities. 
\end{proof} 

\begin{cor}\label{cor:22sympS}
	The blowup $\widetilde{\GChar{\SL_2}{2}}$ of $\GChar{\SL_2}{2}$ along the
	reduced ideal defining the singular locus of $\GChar{\SL_2}{2}$ is a
	smooth symplectic variety and $\GChar{\SL_2}{2}$ has symplectic singularities.
\end{cor}

\begin{proof}
	Let $I$ denote the reduced ideal in $\C[\GChar{\SL_2}{2}]$ defining the
	singular locus. Since the singular locus is stable under the action
	of $\Z_n^{2g}$, so too is $I$. Therefore, the action of $\Z_n^{2g}$
	lifts to the blowup $\widetilde{\GChar{\SL_2}{2}}$ making
	$\sigma : \widetilde{\GChar{\SL_2}{2}} \rightarrow \GChar{\SL_2}{2}$
	equivariant. Theorem \ref{thm:Blowup22LS}, together with the fact that
	$$
	\widetilde{\GChar{\GL_2}{2}} \cong \widetilde{\GChar{\SL_2}{2}} \times_{\Z_{n}^{2g}} T,
	$$
	implies that $\widetilde{\GChar{\SL_2}{2}}$ is smooth. Moreover, the fact that $\widetilde{\GChar{\GL_2}{2}} \rightarrow \GChar{\GL_2}{2}$ is semi-small implies that $\sigma : \widetilde{\GChar{\SL_2}{2}} \rightarrow \GChar{\SL_2}{2}$ is semi-small. The argument that this implies that $\sigma$ is a symplectic resolution is identical to the first part of the proof of Corollary \ref{cor:22symp}. 
\end{proof}
Looking at the proofs of the preceding corollaries, we see that the above blowup provides symplectic resolutions of $\GChar{\GL_2}{2}$ and $\GChar{\SL_2}{2}$. This verifies Theorem \ref{t:constr-res}.(ii).

\subsection{The genus one case}\label{sec:genus1}

Recall that $\G$ equals either $\GL$ or $\SL$. Let $\T$ be a maximal torus in $\G$. The following is well-known. It can be deduced from the corresponding statement for the commuting variety in $\g \times \g$, where $\g = \mr{Lie} \, \G$; see \cite{JosephChevalley}, \cite[Sections 2.7 and 2.8]{AlmostCommutingVariety}. 

\begin{prop}\label{prop:genusone}
Fix $g = 1$. As symplectic singularities, the $\G$-character variety of $\Sigma$ is isomorphic to $(\T \times \T) / \s_n$.   
\end{prop}



Unlike the case $g > 1$, it is not clear whether $\Hom(\pi,\G)$ is reduced, but it is shown in  \cite{AlmostCommutingVariety} that the corresponding $\G$-character variety is reduced. This is the main difficulty: the above statement on the level of the reduced character variety was proved earlier by Joseph \cite{JosephChevalley}. In the case $\G = \GL$, the Hilbert--Chow morphism defines a symplectic resolution $\pi : \Hilb^n(\Cs \times \Cs) \rightarrow (\T \times \T) / \s_n$. Similarly, the the preimage $\Hilb^n_0(\Cs \times \Cs) \subset \Hilb^n(\Cs \times \Cs)$ of $\GChar{\SL_n}{1} \subset \GChar{\GL_n,1}$ under $\pi$ defines a symplectic resolution of $\GChar{\SL_n}{1}$. Notice that the case $n = 1$ is trivial since $\GChar{\GL_1}{1} = \Cs \times \Cs$ with its standard symplectic structure. 

We  also need one additional case in genus one, aside from general and special linear groups,  which is best treated here:
\begin{prop}\label{prop:genusone-pgl2}
The character variety $\GChar{\PGL_2}{1}$ admits a projective resolution of singularities, given by the blowup of $\Hilb^2_0(\Cs \times \Cs)/\Z_2^2$ along the singular locus.
\end{prop}
\begin{proof}
  Observe that $\PGL_2 = \SL_2 / \Z_2$. Hence, $\GChar{\PGL_2}{1}= \CharS(\SL_2,1)/\Z_2^2$. Now, the $\Z_2^2$ action lifts to the resolution $\Hilb^2_0(\Cs \times \Cs) \to \CharS(\SL_2,1)$, acting by negation in each of
  the factors of $\Cs$. Then we have the birational projective Poisson morphism
  $\Hilb^2_0(\Cs \times \Cs)/\Z_2^2 \to \GChar{\PGL_2}{1}$.

  We claim that the singularities on the source are all type $A_1$ surface singularities, hence reduced. This is because, writing $\Z_2 = \{-I,I\} \in \SL_2$, the fixed locus of $(-I,I) \in \SL_2^2$ consists of the two points $\{(i, 1),(-i, 1)\}, \{(i,-1), (-i, -1)\} \in \Hilb^2(\Cs \times \Cs)$, which do not include any of the fixed points for
  $(I,-I)$ (namely, the two points $\{(\pm 1, i), (\pm 1, -i)\} \in \Hilb^2(\Cs \times \Cs)$) or for
  $(-I,-I)$ (the two points $\{(i, \pm i), (-i, \mp i)\} \in \Hilb^2(\Cs \times \Cs)$). Thus the nontrivial isotropy groups of points on $\Hilb_0^2(\Cs \times \Cs)$ are all of size two, at which we get type $A_1$ surface singularities. This is also \emph{a priori} clear since $\Z_2^2$ acts symplectically on the surface $\Hilb_0^2(\Cs \times \Cs)$, and so the linearization of the isotropy group of a point is isomorphic to a subgroup of $\Sp_2=\SL_2$, which cannot be $\Z_2^2$.

 Now,  since  $\Hilb_0^2(\Cs \times \Cs)/\Z_2^2$ only has type $A_1$ surface singularities,  the blowup of the singular locus is a symplectic resolution, which by composition gives a symplectic resolution of $\GChar{\PGL_2}{1}$.  
\end{proof}
This  immediately implies Theorem \ref{t:constr-res}.(i).

\subsection{Factoriality}\label{sec:localfactchar}

In this section we show that the character variety is locally factorial when $n,g \ge 2$ and $(n,g) \neq (2,2)$. We begin with the $\GL$-character variety. Recall that $\xi : \Hom(\pi,\GL_n) \rightarrow \GCharg{\GL_n}$ is the quotient map. We begin with the following dimension estimates of Crawley-Boevey and Shaw:

\begin{thm} \cite[Theorem 7.2, Corollary 7.3]{CBShaw}\label{t:cbnormalChar} 
	Consider a stratum $C_{\nu}$ in $\GChar{\GL_n}{g}$ of representation type
	$(\ell_1, \nu_1; \ds ; \ell_k,\nu_k)$. Then for all $z \in C_{\nu}$, the fibre $\xi^{-1}(z) \subseteq \Hom(\pi,\GL_n)$ has dimension at most $n^2 g - \sum_t ( \nu_t^2 (g-1) +1)$, so the dimension of $\xi^{-1}(C_{\nu})$ is at most $n^2 g + \sum_t ( \nu_t^2 (g-1) +1)$.
\end{thm}

The action of $\GL$ on $\Hom(\pi,\GL_n)$ factors through $\PGL_n$. The open subset of $\Hom(\pi,\GL_n)$ where $\PGL_n$ acts freely is denoted $\Hom(\pi,\GL_n)_{\free}$.

\begin{lem}\label{lem:freecompcharvariety}
  Assume that $n,g > 1$ and $(n,g) \neq (2,2)$. The variety $\Hom(\pi,\GL_n)$ is normal and locally factorial. The complement to $\Hom(\pi,\GL_n)_{\free}$ in $\Hom(\pi,\GL_n)$ has codimension at least four.
\end{lem}

\begin{proof}
By Theorem~\ref{thm:characterproperties}(2), $\Hom(\pi,\GL_n)$ is a complete intersection and hence Cohen-Macaulay. Thus, it satisfies $(S_2)$. By a theorem of Grothendieck, \cite[Theorem 3.12]{KaledinLehnSorger}, in order to show that $\Hom(\pi,\GL_n)$ is locally factorial, it suffices to check that it satisfies $(R_3)$ too. We need to check the second claim that the complement to $\Hom(\pi,\GL_n)_{\free}$ in $\Hom(\pi,\GL_n)$ has codimension at least	four.

Note that the free locus is the same as the locus of representations whose endomorphism algebra has dimension one. These are the ``bricks''. We represent $\Hom(\pi,\GL_n)$ as the union of preimages of the (finitely many) strata, and consider over each such preimage the non-free locus. 
	
If the preimage of the stratum has codimension at least four, it can be ignored. Thus, we just need to show that the complement to $\xi^{-1}(C)_{\mathrm{free}}$ has codimension at least four for those strata $C$ with $\mathrm{codim} \ \overline{\xi^{-1}(C)} \le 3$. Since we are explicitly excluding the case $(n,g) = (2,2)$, Theorem~\ref{t:cbnormalChar}, together with Lemma~\ref{lem:stratadimChar}(3), imply that we are reduced to considering the cases $(n,g) = (3,2)$ and $\nu = (1,2;1,1)$, or $(n,g) = (2,3)$ and $\nu = (1,1;1,1)$.     

The semi-simple part of a fibre $\xi^{-1}(M)$ is the $\GL$-orbit of $M$, which has dimension $n^2 - \dim \End(M)$. If $\xi^{-1}(C)_{\mathrm{ss}}$ denotes the semi-simple part of the preimage of the stratum $C$ then
$$
\mathrm{codim}_{\Hom(\pi,\GL_n)} \, \xi^{-1}(C)_{\mathrm{ss}} = \dim \Hom(\pi,\GL_n) - \dim C -(n^2 - \dim \End(M))
$$
and the difference $\mathrm{codim}_{\Hom(\pi,\GL_n)} \, \xi^{-1}(C)_{\mathrm{ss}} - \mathrm{codim}_{\GChar{\GL_n}{g}} \,  C$ equals 
$$
 2 n^2 g - (n^2 -1) - \dim C - (n^2 - \dim \End(M)) - (2n^2(g-1) + 2 - \dim C) = \dim \End(M) - 1. 
$$
That is, the semi-simple part of the preimage $\xi^{-1}(C)$ has codimension (in $\Hom(\pi,\GL_n)$) at least the codimension of $C$ itself (in $\GCharg{\GL_n}$). Thus if $C$ has codimension at least four (which is the case for us), then we can ignore the semi-simple part of $\xi^{-1}(C)$.
	
Next, if we consider a stratum $C$ of type $(1,a;1,b)$, note that every representation in this stratum is either semi-simple or an indecomposable extension of two non-isomorphic representations. The latter type is a brick, since there is a unique simple quotient and a unique simple subrepresentation and the two are nonisomorphic. Therefore applying the previous paragraph together with Lemma~\ref{lem:stratadimChar}(2) shows that we can ignore $\xi^{-1}(C)$ (the non-free locus has overall codimension at least four). This proves the final assertion.
\end{proof}

If $\Hom(\pi,\SL)_{\free}$ denotes the open subset of $\Hom(\pi,\SL)$ where $\PGL$ acts freely then Lemma~\ref{lem:charSfiberbundle} implies that the statement of Lemma~\ref{lem:freecompcharvariety} holds for $\Hom(\pi,\SL)_{\free}$ too. 

\begin{proof}[Proof of Theorem \ref{thm:slsqfactor} for $\GCharg{\SL_n}$ and $\GCharg{\GL_n}$]
  Let $X := \GCharg{\SL_n}$ or $\GCharg{\GL_n}$.
  First we suppose that  $n,g > 1$ and $(n,g) \neq (2,2)$.
  Let $X_s$ denote the dense open subset of $X$ consisting of simple representations and $\Hom(\pi,\G)_s$ its preimage in $\Hom(\pi,G)$. Then $\xi : \Hom(\pi,G) \rightarrow X_s$ is a principal $\mathrm{PGL}$-bundle. Moreover, by Lemma \ref{lem:stratadimChar}(2), the complement to $X_s$ has codimension at least $4$ in $X$. We are therefore in the situation where we can apply the results of Dr\'ezet's Theorem to $X$. 

The stratum $C_{\rho}$ of type $\rho = (n,1)$ is contained in the closure of all other strata in $X$. If $y \in T(x)$ is a lift in $\Hom(\pi,G)$ of a point $x$ of $C_{\rho}$ then $y$ corresponds to the representation $\C^{\oplus n}$, where $\C$ denotes here the trivial $\pi$-module. Therefore $\PGL_y = \PGL_n$ has no non-trivial characters. In particular, $\PGL_y$ will act trivially on $L_y$ for any $\PGL$-equivariant line bundle on $\Hom(\pi,G)$. Hence, we deduce from Dr\'ezet's Theorem \cite[Th\'eor\`em A]{Drezet} that $X$ is factorial at every point of the closed stratum $C_{\rho}$.

Now consider an arbitrary stratum $C_{\tau}$ in $X$. If $X$ is factorial at one point of the stratum then it will be factorial at every point in the stratum. On the other hand, the main result of \cite{OpenFactorial} says that the subset of factorial points of $X$ is an open subset. Since this open subset is a union of strata and contains the unique closed stratum, it must be the whole of $X$. This completes the proof in these cases.

In the case that $(n,g)=(2,2)$, as observed in \cite[Remark
4.6]{KaledinLehn}, see also \cite{KaledinLehnSorger} and
\cite[Proposition 7.3]{Fu-Qfactor}, the varieties $\GCharg{\GL_n}$ and
$\GCharg{\SL_n}$ are $\mathbb{Q}$-factorial (but not locally factorial).  In
the case $g=1$, by Proposition \ref{prop:genusone},
it follows that $\GCharg{\GL_n}$ and $\GCharg{\SL_n}$ are $\mathbb{Q}$-factorial, being finite quotients of smooth (hence locally factorial) varieties (see, e.g.,
 \cite[Section 2.4]{BellSchedler2large} and the references therein).
\end{proof}
\begin{remark}
A similar analysis has been performed in \cite{ST-srmqvcvps} in order to classify which moduli spaces of semi-simple representations of an arbitrary multiplicative deformed preprojective algebra admit symplectic resolutions. 
\end{remark}

\section{Arbitrary type $A$ groups}\label{sec:generaltypeAchar}
Now let $\G$ be connected reductive, all of whose simple quotients are of type $A$.  We begin this section by stating (but not yet proving) a refined version of Theorem \ref{t:class-res}, giving the closest thing to a symplectic resolution that can exist in general. Then we explain a general procedure to reduce statements about $\G$ to those about a finite covering, and hence to the case of $\SL_n$.  This is already enough, thanks to Section \ref{s:gl-sl}, to prove Theorems \ref{thm:slsymp} and \ref{thm:slsqfactor}, which we do in section~\ref{sec:31proof}. Next, we explain how to estimate codimension of fixed loci under translation by finite-order central elements of $\G$.  This, together with the preceding reduction, is enough to prove Proposition \ref{p:codim}.  Finally, we put all these results together and prove 
Theorems \ref{t:rmm} and \ref{t:class-res}.

\subsection{Relative minimal models}\label{ss:min}

Most of our main results are consequences of the more general Theorem~\ref{t:rmm}, which describes for \emph{all}
character varieties of type $A$ the closest thing to a symplectic resolution that can exist. Namely, a \emph{$\mathbb{Q}$-factorial terminalization} $f \colon X \to Y$ of $Y$ is a $\Q$-factorial variety $X$ with terminal
singularities and a crepant projective birational morphism
$f \colon X \to Y$. Such pairs can always be constructed as a special
type of \emph{relative minimal model} thanks to \cite[Corollary
1.4.3]{BCHM}; hence we will sometimes use this terminology. Note that $f$ is actually \emph{maximal} crepant in the sense that any
further crepant proper birational morphism $X' \to X$ is an
isomorphism.

Given a connected reductive group $\G$ we write, as in section~\ref{ss:class}, $\G = (H \times K)/Z$ where $K$ is semi-simple and $Z < H \times Z(K)$ is a finite subgroup with $Z \to H \times K \to H$ injective. Then $\GCharg{\G} = \GCharg{H \times K}/Z^{2g}$, where the $Z^{2g}$ action on the product $\GCharg{H \times K} = \GCharg{H} \times \GCharg{K}$ is free because of the injectivity of $Z \to H$. It will be convenient also to write $K = K_0/C_0$ where $K_0 = \prod_{i=1}^m \SL_{n_i}$ is simply-connected, and $C_0 < Z(K_0)$. Then $\G$ is a finite quotient of the product $\mathbb{G}:=H \times K_0$ of a torus with special linear groups.

Finally, in the situation of genus one, we will also need the
intermediate group
$K_1 = \prod_{i=1}^{m-m_1} \SL_{n_i} \times \PGL_2^{m_1}$, which fits into a chain $K_0 \twoheadrightarrow K_1 \twoheadrightarrow K$, such that $m_1$ is maximal.  In other words, $K_1$ is the quotient of $K_0$ by the subgroup of $C_0$ generated by elements which are $-I_2$ in some factor $\SL_2$ and the identity in other factors. Let $C_1$ be the quotient of $C_0$ such that $K=K_1/C_1$.


Note that $\GCharg{K_0} = \prod_i \GCharg{\SL_{n_i}}$; each factor $\GCharg{\SL_{n_i}}$ either admits a symplectic resolution (when $g=1$ or $(n,g)=(2,2)$ by Theorem \ref{t:constr-res}), or is $\Q$-factorial terminal (by Theorem \ref{thm:slsqfactor} and Proposition \ref{p:codim} for the $\SL_n$ case).
Similarly, in the case $g=1$,
$$
\GChar{K_1}{1} = \prod_{i=1}^{m-m_1} \GChar{\SL_{n_i}}{1} \times \GChar{\PGL_2}{1}^{m_1}.
$$
In this case each factor admits a symplectic resolution by Theorem \ref{t:constr-res}. Thus, by taking symplectic resolutions of factors admitting them, we obtain a $\Q$-factorial terminalization $\widetilde{\GCharg{K_0}} \to \GCharg{K_0}$ for $g > 1$ and $\widetilde{\GChar{K_1}{1}} \to \GChar{K_1}{1}$ for $g=1$. To see that the result is indeed $\Q$-factorial, we note that it is a product of $\Q$-factorial varieties, and apply the main result of \cite{OpenFactorial}.

In general, automorphisms do not lift to $\Q$-factorial terminalizations. However, we claim that a $\Q$-factorial terminalization $\widetilde{\GCharg{\G}}$ of $\GCharg{G}$ may be obtained from the aforementioned one by taking a finite quotient:

\begin{thm} \label{t:rmm} 
  \begin{itemize}
  \item[(i)] If $g=1$,
    a 
    $\Q$-factorial terminalization of $\GChar{K}{1}$ is given by the quotient of $\widetilde{\GChar{K_1}{1}}$ by a (unique) lift of the action of $C_1^{2}$.
      \item[(ii)] If $g \geq 2$, a $\Q$-factorial terminalization of $\GCharg{K}$
    can be obtained as the quotient of $\widetilde{\GCharg{K_0}}$ 
    by a (unique) lift of the action of $C^{2g}_0$.
  \end{itemize}
In both cases, given a $\Q$-factorial terminalization $\widetilde{\GCharg{K}} \to \GCharg{K}$, then a $\Q$-factorial terminalization $\widetilde{\GCharg{\G}}$ of $\GCharg{\G}$ can be obtained as the \'etale quotient $(\GCharg{H} \times \widetilde{\GCharg{K}})/Z^{2g}$ (with the same singularities as $\widetilde{\GCharg{K}}$).
\end{thm}



\subsection{Reduction to $\SL$ case}
\label{ss:fin-quotient}
There exists a product of special linear groups and a torus $H$ such that $\G$ is the quotient of $\mathbb{G} := H \times \prod_i \SL(n_i,\C)$ by a finite central subgroup, call it $Z_0$ (to distinguish from the group $Z$ appearing in Section \ref{ss:class}, which is the quotient of $Z_0$ by $C_0$).

We now have the following general lemma:
\begin{lem}\label{lem:char-central-quotient}
  Let $G_0$ be an algebraic (or topological) group and $\phi: G_0 \to G_1$ a surjective homomorphism with central finite (or discrete) kernel $Z_0$. Then we have a $G_1$-equivariant isomorphism $\Hom(\pi,G_1)^\circ \cong \Hom(\pi, G_0)^\circ/Z^{2g}_0$.  
  \end{lem}
  \begin{proof}
    Let $F$ be the free group on $2g$ generators. Then $\Hom(\pi,Z_0) = \Hom(F,Z_0) = Z^{2g}_0$ because the abelianization of $\pi$ and $F$ are both $\Z^{2g}$.
Any    homomorphism $\pi \to G_0$ can be lifted arbitrarily on generators $x_i, y_i$ of $F$, so that the image of the relation $\prod [x_i, y_i]$ is some element $z \in Z_0$. This defines
a connected component $R_z$ of $R:=\Hom(\pi, G_0)$, namely the image of the subset of $\Hom(F, G_0) \cong G^{2g}_0$ where $\prod_i [x_i, y_i] \mapsto z$, under composition with $\phi$.  Note that composition with $\phi$ identifies with the quotient map by the free action of $Z^{2g}_0$ on $G^{2g}_0$. The connected component of the trivial homomorphism of $\Hom(\pi, G_0)$ is then $R_1 = \Hom(\pi, G_0)^\circ/Z^{2g}_0$. Finally, since $Z_0$ is central in $G_0$, we have that the adjoint action of $G_0$ on $\Hom(\pi,G_0)$ commutes with the action of $Z^{2g}_0$. On the quotient $\Hom(\pi,G_0)/Z^{2g}_0$, this factors to an adjoint action of $G_1$. Then by construction the adjoint actions of $G_1$ on both sides of the isomorphism match.
\end{proof}
Specialising to our situation we have:

\begin{cor}\label{cor:sympZfinitequot}
	 With notation as above, $\GCharg{\G} \cong \GCharg{\mathbb{G}}/Z_0^{2g}$. The action of $Z^{2g}_0$ on $\GCharg{\mathbb{G}}$ is symplectic.  
\end{cor}
\begin{proof}
	Note that $\Hom(\pi,\mathbb{G})$ is connected. By Lemma~\ref{lem:char-central-quotient}, we have established
  $\Hom(\pi,\G)^\circ\cong \Hom(\pi,\mathbb{G})/Z_0^{2g}$, compatibly with the adjoint actions.
  Therefore, taking the
  adjoint quotients, we obtain the desired isomorphism. The fact that
  the action of $Z^{2g}_0$ is symplectic follows from the
  construction of \cite{GoldmanSymplectic}, since the adjoint action
  of $Z_0$ on the Lie algebra of $\mathbb{G}$ is trivial.
  \end{proof}

  \begin{remark}
    More generally, the proof of Lemma \ref{lem:char-central-quotient} above shows that, for every connected component $X \subseteq \Hom(\pi, G_0)$, then $X/Z^{2g}_0$ is a connected component of $\Hom(\pi, G_1)$, which is the image of $X$ under composition with the surjection $\phi \colon G_0 \to G_1$. Still more generally, we can replace $\pi$ by a quotient $F / K$ of a free group $F$ on $\ell$ generators by any subgroup $K < [F,F]$ generated by commutators.  Then every connected component of $\Hom(F/K, G_1)$ maps, under composition with $\phi$, onto a connected component of $\Hom(F/K, G_1)$, and this map is the quotient map by the free $Z^\ell_0$ action.  
    \end{remark}
\begin{remark}An alternative proof of  Lemma \ref{lem:char-central-quotient} can be given using group cohomology:
   the exact sequence $1 \to Z_0 \to G_0 \to G_1 \to 1$ can be viewed as an exact sequence of nonabelian $\pi$-modules with trivial $\pi$-action. Since $Z_0$ is abelian,
the long exact sequence for nonabelian cohomology extends to $H^2(\pi,Z_0)$. Since the action is trivial, on zeroth cohomology this includes the original exact sequence. As $H^1(\pi,M)=\Hom(\pi,M)$ for $M$ a module with trivial action, we obtain the exact sequence:
$$
1 \to \Hom(\pi,Z_0) \to \Hom(\pi,G_0) \to \Hom(\pi,G_1) \to H^2(\pi,Z_0).
$$
Since $H^2(\pi,Z_0)$ is discrete, the image of
$\Hom(\pi,G_1)^{\circ}$ under the connecting morphism
$\Hom(\pi,G_1) \to H^1(\pi,Z_0)$ is trivial.
Thus
$\Hom(\pi,G_1)^{\circ} = \Hom(\pi,G_0)^\circ/Z^{2g}_0$.
\end{remark}


 
\subsection{Proof of Theorems \ref{thm:slsymp} and \ref{thm:slsqfactor}}\label{sec:31proof}
In Section \ref{s:gl-sl}, we proved these theorems for the cases $\GCharg{\GL_n}$ and $\GCharg{\SL_n}$. They therefore hold for the product $\mathbb{G} = H \times \prod_i \GCharg{\SL_{n_i}}$ as follows. The product of symplectic singularities is clearly symplectic.  For the $\mathbb{Q}$-factorial property, this not obvious but proved in \cite{OpenFactorial}. In the special case when $g > 2$,  or the case $g=2$ and $\mathbb{G}$ has no factors of the form $\SL_2$, we can also use the argument given in Section \ref{s:gl-sl}---deducing factoriality from a quotient---verbatim to deduce the property for $\mathbb{G}$.

By  \cite[Proposition 2.4]{Beauville}, the finite  symplectic quotient $\GCharg{\G}$ of $\GCharg{\mathbb{G}}$ also has symplectic singularities.  This result is also known for the $\mathbb{Q}$-factorial property: see \cite[Section 2.4]{BellSchedler2large} and the references therein.


\subsection{Codimension of fixed points: genus $g \geq 2$}
%
We consider the action of $Z^{2g}_0$ on the open subset
$\CharS(\mathbb{G},g)_{\text{irr}}$ of $\GCharg{\mathbb{G}}$
consisting of products of irreducible representations. Since
$\GCharg{\mathbb{G}}$ is a product of the character varieties of its
factors, it is irreducible by Theorem
\ref{thm:characterproperties}(1), Lemma \ref{lem:charSfiberbundle},
and Proposition \ref{prop:genusone}.  Thus, if nonempty,
$\GCharg{\mathbb{G}}_{\text{irr}}$ is dense.  It is nonempty if $g > 1$
by Theorem \ref{thm:characterproperties}(3): given an irreducible
representation in $\GCharg{\GL_n}$, by rescaling generators we get an irreducible representation in $\GCharg{\SL_n}$, i.e., a representation of
$\GCharg{\GL_n}$ is irreducible if and only if the point of its fiber of
$\GCharg{\GL_n} \to T$ in Lemma \ref{lem:charSfiberbundle} is irreducible as
  an element of $\GCharg{\SL_n}$. 
  
Note that, for $g=1$, $\GCharg{\mathbb{G}}_{\text{irr}}$ is nonempty if and only if $\mathbb{G}$ is abelian, otherwise the generic
  representation of $\Z^2 \to \SL_n$ for $n > 1$ is a direct sum
  of one-dimensional representations.

\begin{prop}\label{prop:fixedsmoothcodim4}
  The set of points in $\CharS(\mathbb{G},g)_{\mathrm{irr}}$ with non-trivial stabilizer under the action of $Z^{2g}_0$ has codimension at least $4(g-1)$.
\end{prop}

In particular, for $g \ge 2$, the codimension is at least four. In the proof we find the precise codimension: see Remark \ref{rem:precise-codim} below.

\begin{proof}
  It suffices to show that for each non-trivial $\sigma \in Z^{2g}_0$, the set of points in $\CharS(\mathbb{G},g)_{\text{irr}}$
  fixed by $\sigma$ has codimension at least $2(g-1)$. Let $\rho \colon \pi \to \mathbb{G}$ with $[\rho] \in \CharS(\mathbb{G},g)_{\text{irr}}$ and assume $[\rho]$ is fixed by $\sigma$. Then $\sigma \cdot \rho \cong \rho$. This implies that there exists $A \in \mathbb{G}$ such that $A^{-1} \rho(f) A = \sigma(f) \rho(f)$ for all $f \in \pi$.

  Let $\ell \geq 2$ be the order of $\sigma$.
Note that, since $\sigma^\ell = 1$, the element $A^\ell$
        is an endomorphism of $\rho$. By assumption this implies that $A^\ell$ is a scalar matrix in each factor of $\mathbb{G}$.  As a result, up to rescaling $A$ in each factor, we can assume $A^\ell=1$, and hence is semi-simple with eigenvalues $\ell$-th roots of unity.

The element $\sigma$ acts on the tangent space of $\CharS(\mathbb{G},g)_{\text{irr}}$ at $[\rho]$ and it suffices to show that the codimension of $(T_{[\rho]} \CharS(\mathbb{G},g)_{\text{irr}})^{\sigma}$ is at least $4$ in $T_{[\rho]} \CharS(\mathbb{G},g)_{\text{irr}}$. To compute the action of $\sigma$ on $T_{[\rho]} \CharS(\mathbb{G},g)_{\text{irr}}$, we twist the action so that $\rho \in \Hom(\pi,\mathbb{G})$ is a fixed point. Namely, define $\tilde{\sigma}$ acting by $(\tilde{\sigma} \tau)(x) = \sigma(x) A \tau(x) A^{-1}$ for all $\tau \in \Hom(\pi,\mathbb{G})$. Then $\tilde{\sigma} \rho = \rho$.

Let $\mf{g}$ be the Lie algebra of $\mathbb{G}$. Let $\mf{g}_\rho$ denote the
representation of $\pi$ obtained by the action $\gamma \cdot x := \Ad(\rho(\gamma))(x)$.
As noted in \cite{GoldmanSymplectic}, we can identify $T_{[\rho]} \CharS(\mathbb{G},g)_{\text{irr}} = H^1(\pi,\mf{g}_\rho)$. Let $[u] \in H^1(\pi,\mf{g}_\rho)$ be represented by the $1$-cocycle $u$. As in \cite[(1.1)]{GoldmanSymplectic}, the tangent vector $[u]$ defines an infinitesimal curve through $\rho$ by $\rho_t(x) = \exp(t u(x)) \rho(x)$. Then 
\begin{align*}
(\tilde{\sigma} \rho_t)(x) & = \sigma(x) A \exp(tu(x)) \rho(x) A^{-1} \\
& = \sigma(x) A \exp(tu(x)) A^{-1} A \rho(x) A^{-1} \\
& = \exp(t A u(x) A^{-1}) \sigma (x)  A \rho(x) A^{-1} \\
& = \exp(t A u(x) A^{-1}) \rho(x),    
\end{align*}
where we have used the fact that $\sigma(x)$ is central. Thus, the action of $\sigma$ on $H^1(\pi,\mf{g}_\rho)$ is given by conjugation by $A$ on
$\mf{g}$. We use the ingenious argument in \cite[Section 1.5]{GoldmanSymplectic} to estimate $\dim H^1(\pi,\mf{g}_\rho)^A$. Since $\Sigma$ is $K(\pi,1)$, $\mf{g}_\rho$ defines a local system on $\Sigma$ of rank $\dim \mf{g}$ with Euler characteristic 
$$
\sum_{i = 0}^2 (-1)^i \dim H^i(\pi,\mf{g}_\rho) = -2(g-1) \dim \mf{g}. 
$$
This formula follows from the fact that, as a "virtual" vector space,
$$
\sum_{i = 0}^2 (-1)^i H^i(\pi,\mf{g}_\rho) = \sum_{i = 0}^2 (-1)^i C^i \o_{\C} \mf{g}
$$
where $C^{\idot}$ are the terms of the (finite-dimensional) simplicial cochain complex computing the cohomology of $\pi$ with coefficients in $\C$. It follows that 
$$
\sum_{i = 0}^2 (-1)^i H^i(\pi,\mf{g}_\rho)^A = \sum_{i = 0}^2 (-1)^i C^i \o_{\C} \mf{g}^A
$$
and hence 
$$
\sum_{i = 0}^2 (-1)^i \dim H^i(\pi,\mf{g}_\rho)^A = -2(g-1) \dim \mf{g}^A.
$$
As noted in Section 1.4 of \cite{GoldmanSymplectic}, $H^0(\pi,\mf{g}) = \End_{\mf{g}}(\mf{g}_{\rho})$ is the infinitesimal stabilizer of $\rho$. The fact that $V$ is irreducible implies that $H^0(\pi,\mf{g}) \cong \mf{z}(\mf{g})$, the centre of $\mf{g}$. This is the Lie algebra of $H^{2g}$, of dimension $2g \dim H$. By Poincar\'e duality, $\dim H^2(\pi,\mf{g}) = 2 g \dim H$ too. Thus,
\begin{align*}
\dim H^1(\pi,\mf{g})^A & = 2(g-1) \dim \mf{g}^A + \dim H^0(\pi,\mf{g})^A + \dim H^2(\pi, \mf{g})^A \\
 & \le 2(g-1) \dim \mf{g}^A + 4 g \dim H;
\end{align*}
in fact we have equality since $\mf{z}(\mf{g})^A = \mf{z}(\mf{g})$ (but we do not need to use this).
Thus, the codimension of $(T_{[\rho]} \CharS(\mathbb{G},g)_{\text{irr}})^{\sigma}$ in $(T_{[\rho]} \CharS(\mathbb{G},g)_{\text{irr}})^{\sigma}$ is (at least) $2(g-1)(\dim \mf{g} - \dim \mf{g}^A)$.
Conjugation by $A$ is non-trivial on at least one simple summand.
Such a summand is of the form $\mf{sl}_m$ for some $m \geq 2$. Therefore, for the dimension estimate we may assume $\mf{g} = \mf{sl}_m$. If the multiplicities of the $\ell$-th roots appearing on the diagonal of $A$ are $m_1, \ds, m_{\ell}$, where $m = m_1 + \cdots + m_{\ell}$, then an easy induction shows that
$$
\dim \mf{sl}_m - \dim \mf{sl}_m^A = m^2 - \sum_{i = 1}^{\ell} m_i^2 \ge 2(t-1),
$$
where $t$ is the number of $m_i$ not equal to $0$. Since $A$ is not the identity matrix, $t \ge 2$. This implies
that $2(g-1)(\dim \mf{sl}_m - \dim \mf{sl}_m^A) \ge 2(g-1)$ as required. 	
\end{proof}

The argument of the proof above does not make use of the product decomposition for $\mathbb{G}$, only for its Lie algebra $\mathfrak{g}$. Hence it is valid with $\mathbb{G}$ replaced by any connected reductive group with type $A$ quotients, and $Z$ any finite central subgroup, replacing the irreducible locus by the locus of representations whose endomorphisms identify with $\mathfrak{z}(\mathfrak{g})$.

\begin{remark}\label{rem:precise-codim}
  We can actually compute the precise codimension as follows. Let $\sigma \in Z^{2g}_0$ be nontrivial. If $\sigma$ has a nontrivial projection to $H$, then it is clear that $\sigma$ has no fixed points, as $H$ is abelian.  So assume that this is not the case.   Let $\sigma_i$ be the
  projection of $\sigma$ to  each factor $\SL(n_i,\C)$, and let $\ell_i \geq 1$ be the minimum positive integer such that $\sigma_i^{\ell_i}$ is a scalar (i.e., the minimum order of a scalar multiple of $\sigma_i$).  Then we claim that:
  $$
  \codim_{\CharS(\mathbb{G},g)_{\mathrm{irr}}} \CharS(\mathbb{G},g)_{\mathrm{irr}}^\sigma = 2(g-1)\sum_i n_i^2(1 - 1/\ell_i).
  $$
  We prove this claim below. 
  In particular, this is at least $4(g-1)$.  The codimension of the nonfree locus under all of $Z^{2g}_0$ is the minimum of the above codimensions; this recovers the statement of Proposition \ref{prop:fixedsmoothcodim4}.

  To prove the claim,  note that it suffices to consider the case of a single factor $\SL(m,\C)$. Up to scaling, as in the proof of the proposition, assume that $\sigma = A \in \SL_n$ has order $\ell$.  Let $\zeta$ be a primitive $\ell$-th root of unity.  
  Conjugating by a permutation matrix, we can assume that for some $m_0, \ldots, m_{\ell-1}$, the first $m_0$ diagonal entries of $A$ are $\zeta^0=1$, the next $m_1$ are $\zeta^1$, the next $m_2$ are $\zeta^2$, and so on.
  Now let $X \in \GL_n(\C)$ be a matrix such that 
  $AXA^{-1} = \lambda X$ for some $\lambda \in \C$. Then $\lambda = \zeta^k$ for some $k \mid \ell$. Then, $X$ must be a block permutation matrix with respect to the decomposition $m=m_0+\cdots+m_{\ell-1}$, expressing the permutation $i \mapsto i+k \pmod{\ell}$.
  For this to be invertible, the blocks must be square, which implies $m_i=m_{i+k}$ for all $i$, with addition taken modulo $\ell$. Now, if $X_1, \ldots, X_{2g}$ are the generating matrices, then we claim that the powers $k_1, \ldots, k_{2g}$ appearing in this calculation must have gcd equal to one. Otherwise, their least common multiple $\ell'$ would be a proper factor of $\ell$, and $A^{\ell'}$ would be a nonscalar automorphism of the representation---impossible by Schur's Lemma. Thus since $m_i = m_{i+k_j}$ for all $i,j$, we get that $m_0 = m_1 = \cdots = m_{k-1}$.  We thus conclude:
  $$
  \dim \mf{sl}_m - \dim \mf{\sl}_m^A = m^2 - \ell(m/\ell)^2 = m^2(1 - 1/\ell).
  $$
  Taking the sum of codimensions over each factor, we obtain the claim.

\end{remark}
The above allows us to prove Proposition \ref{p:codim}:
\begin{proof}[Proof of Proposition \ref{p:codim}]
  Let us first assume that $g=1$ and $\G$ is nonabelian.  We claim that the singular locus has codimension at most two; since $\GCharg{\G}$ is a symplectic singularity, this implies immediately that the codimension is exactly two. Thanks to Corollary~\ref{cor:sympZfinitequot}, it suffices to assume that $\G=\mathbb{G}$. In turn this reduces to the case of $\GChar{\SL_n}{1}$ for $n \geq 2$. Now the statement follows from Proposition \ref{prop:genusone}.

Next assume that $g=2$ and that $\G$ has $\PGL_2$ as a quotient.  We claim again that the singular locus has codimension at most two (hence exactly two). As before, this reduces to the case of $\GChar{\SL_2}{2}$.  Then the statement follows from Lemma \ref{lem:stratadimChar}(1) (see also Section \ref{ss:22}).

Finally, assume that $g \geq 2$ and, if $g=2$, then there are no simple quotients of $\G$ of the form $\PGL_2$. Then the singular locus of
$\GCharg{\mathbb{G}}$ has codimension at least four by Lemma  \ref{lem:stratadimChar}. Its complement, the smooth locus of $\GCharg{\mathbb{G}}$, equals $\GCharg{\mathbb{G}}_{\mathrm{irr}}$ by Lemma \ref{lem:singlocusGoldman}. Therefore, the singular locus of $\GCharg{\G} = \GCharg{\mathbb{G}}/Z_0^{2g}$ is the image under the quotient by $Z^{2g}_0$ of the union of the singular locus of $\GCharg{\mathbb{G}}$ and the non-free locus. Thus we can conclude by Proposition \ref{prop:fixedsmoothcodim4} that the singular locus has codimension at least four.
\end{proof}

\subsection{Genus one case}

We estimate the codimension of the singular locus in genus one:
\begin{prop}\label{p:codim-genus1} Let $\sigma \in Z_0^{2}$ be non-trivial. 
  The set of points in $\GChar{\mathbb{G}}{1}$ with non-trivial stabilizer under the action of $\sigma$ has codimension at least $4$ unless $\sigma$ is equal to minus the identity in a factor of the form $\SL_2$ and the identity in other factors, in which case this locus has codimension two.
\end{prop}
\begin{proof}
  The non-free locus of  $\sigma$ on $\GChar{\mathbb{G}}{1}$ is the product of such loci on each factor $\GChar{H}{1}$ and $\GChar{\SL_{n_i}}{1}$. Therefore the codimension is the sum of these codimensions.  Note that these codimensions are all even.  So, in order to not have codimension four, $\sigma$ must have trivial projection to all factors but one, and there the non-free locus must have codimension exactly two.  This reduces the question to the case of a single factor.  Note that if $H$ is this factor, the non-free locus is empty, so cannot have codimension two. We have reduced to the statement that the non-free locus of the action of a nontrivial element $\sigma \in Z(\SL_n)^{2}$ on $\GChar{\SL_{n}}{1}$ has codimension two if and only if $n=2$. 

  To prove this, let $\sigma= (\zeta I_n,\eta I_n) \in \SL_n^2$, for $\zeta^n = \eta^n = 1$. Let $T < SL_n$ be the maximal torus of diagonal elements. 
  Via the isomorphism $\GChar{\SL_n}{1} \cong (T \times T)/S_n$ of Proposition \ref{prop:genusone}, we see that the action of $\sigma$ is by multiplication by $(\zeta,\eta)$.  Let us think of an element of $(T \times T)/S_n$ as a multiset of $n$ elements $(a_i,b_i) \in \Cs \times \Cs$, with $\prod_i a_i = \prod_i b_i = 1$. Then $(\zeta,\eta) \cdot (a_i,b_i) = (\zeta a_i, \eta b_i)$. 
  
  For $\sigma$ to fix the multiset, it must be partitioned into orbits in $\Cs \times \Cs$ under the action of $\sigma$. Each orbit is uniquely determined by any one of its elements, which can be arbitrary so long as the overall product of pairs is $(1,1)$. The latter condition can be interpreted to mean that all orbits except for one are arbitrarily chosen, with the final one having only finitely many possible choices. Therefore, the dimension of the fixed locus is $2(m-1)$, where $m$ is the number of orbits. Hence the codimension of the fixed locus is $2(n-m)$. For this to be two, we need $m=n-1$. However, every orbit is nontrivial, since $T \times T$ itself has no fixed points under $\sigma$. Therefore, this happens precisely if $n=2$ and $m=1$.
\end{proof}

\subsection{
  Proof of Theorem \ref{t:rmm}}


In general, automorphisms of a symplectic singularity do not lift to a symplectic resolution. However, this is the case in our situation when the automorphisms in question come from the (symplectic) action of $Z^{2g}$. More precisely, we have seen that there are three situations where the character variety associated to an  almost simple group of type $A$ admits a symplectic resolution: (a) $\GChar{\SL_2}{2}$, (b) $\GChar{\SL_n}{1}$, and (c) $\GChar{\PGL_2}{1}$. In (a), any automorphism of $\GChar{\SL_2}{2}$ lifts to the blowup of the singular locus, since the latter is always fixed by the automorphism. In case (b) of $\GChar{\SL_n}{1}$, every automorphism coming from $Z(\SL_n)^{2} \cong \Z_n^{2}$ lifts to an automorphism of the resolution, by the universal properties of the Hilbert--Chow morphism. This is because the action of $\Z_n^2$ is induced from an action on $\Cs \times \Cs$; see the proof of Proposition~\ref{p:codim-genus1}. In both cases, since the resolution is birational the lifts are necessarily unique. Finally, in case (c), the symplectic resolution is given by blowing up the $A_1$ surface singularities as in Proposition \ref{prop:genusone-pgl2}, but the group $Z^{2g}$ is trivial in this case.  

Let $Y$ be the variety occurring in one of (a)--(c), and $f \colon X \to Y$ the symplectic resolution described above. Let $\sigma$ be an automorphism coming from the action of a central (i.e., scalar) element of $\SL_n$ (or $\sigma = I_2$ when the group is $\PGL_2$). Then $\sigma$ lifts to $X$ with $f$ equivariant. 



\begin{lem}\label{l:fix-imm}
  Let $C \subseteq X^\sigma$ be a connected component.  At generic $x \in C$, the restriction $f|_C: C \to f(C)$ is an immersion.
\end{lem}
Since $f(X^\sigma) \subseteq Y^\sigma$, we obtain the following immediate consequence:
\begin{cor}\label{c:fix-dim}
  We have the inequalities $\dim X^\sigma \leq \dim Y^\sigma$.
\end{cor}
\begin{proof}[Proof of Lemma \ref{l:fix-imm}]
 Let $y = f(x)$.  We induct on the codimension of the symplectic leaf through
   $y \in Y$. Over the open leaf (codimension zero), $f$ is an
   isomorphism, so the statement is clear. We proceed to the inductive
   step. Note that our base variety, $Y:=\GCharg{\SL_n}$ (for $g=1$ or
   $(n,g)=(2,2)$), has finitely many symplectic leaves. This is a
   general fact about symplectic varieties \cite[Theorem
   2.3]{Kaledinsympsingularities}, but in our case can also be
   seen explicitly (see the end of Section \ref{ss:ss}). Thus,
   $Y^{\sigma}$ intersects only finitely many symplectic leaves.
   Now suppose inductively that $y \in Y^{\sigma}$ lies in a
   symplectic leaf $L$ of codimension $2k > 0$.  Since
   $X^{\sigma}$ is smooth symplectic, its codimension
   is locally constant. By genericity of $x$, we can assume that nearby points
   in $C$ do not map to a symplectic leaf in $Y$ of smaller codimension. Hence, nearby points also map to $L$.  That is, a neighborhood of $x$ in $C$ maps entirely to $L$.  Now, $f|_{f^{-1}(L)}: f^{-1}(L) \to L$ is a fibration, thanks to the infinitesimal automorphisms of $f: X \to Y$ given by Hamiltonian vector fields.  Moreover, the fibers of $f$ are all isotropic
   (this is actually true for an arbitrary projective
   symplectic resolution thanks to \cite[Proposition
   1.2.2]{GinzburgQuiverVarieties}, based on
   \cite{Wierzba-contractions}). Since $C \cap f^{-1}(L)$ is a symplectic manifold, as is its image---a component of $L^\sigma$---we conclude that its fibers are in fact zero-dimensional. Thus, $f$ is an immersion at $x$, as desired.
 \end{proof}
 \begin{remark} The proof shows that
   $f$ is an immersion everywhere on $C \cap f^{-1}(L)$, so one can replace ``generic $x$'' in the statement of the lemma with ``for $x \in C$ whose image in $Y$ has maximal rank under the Poisson structure''.
\end{remark}

We can construct a $\Q$-factorial terminalization $\widetilde{\GCharg{K_0}}$, resp. $\widetilde{\GChar{K_1}{1}}$, of $\GChar{K_0}{g}$, resp. of $\GChar{K_1}{1}$, by taking the symplectic resolutions described above for the factors of type (a)--(c), and leaving other factors fixed. We have shown that the action of $C_j^{2g}$ on $\GChar{K_j}{g}$ lifts (uniquely) to the $\Q$-factorial terminalization $\widetilde{\GCharg{K_j}}$. 

We must check that $\widetilde{\GCharg{K_j}}/C_j^{2g}$ is a $\Q$-factorial terminalization of $ \GChar{K_j}{g}/C_j^{2g}= \GChar{K}{g}$. By Corollary~\ref{c:fix-dim}, the codimension of the singular locus of $\widetilde{\GCharg{K_j}}/C_j^{2g}$ is the minimum of the codimension of the non-free locus in $\GCharg{K_j}$ and the codimension of the singular locus of the partial resolution $\widetilde{\GCharg{K_j}}$ (which is at least four).  It follows from our codimension estimates, Proposition~\ref{prop:fixedsmoothcodim4} and Proposition~\ref{p:codim-genus1}, that this codimension is at least four (the final situation described in Proposition~\ref{p:codim-genus1} does not occur precisely because of the definition of $K_1$). Hence the partial resolution $\widetilde{\GCharg{K_j}}/C_j^{2g}$ has terminal singularities by \cite{NamikawaNote}.  By construction and Theorem \ref{thm:slsqfactor}, it is a finite quotient of a $\Q$-factorial variety, hence $\Q$-factorial; see \cite[Section
2.4]{BellSchedler2large} and the references therein. 


The final statement is clear, since a quotient by a free action of a finite group preserves the properties of being $\Q$-factorial and of being terminal. This completes the proof of Theorem~\ref{t:rmm}.

\subsection{Proof of Theorem~\ref{t:class-res}}
This proof is based on showing that the character variety is formally locally conical. (Note that, thanks to \cite{Artin}, we could replace ``formal'' by ``\'etale'' if desired.)

First, the fact that (a) and (b) of Theorem~\ref{t:class-res} are equivalent follows immediately from the final statement of Theorem~\ref{t:rmm}. Moreover, if we are in one of the cases in (c) then it is a consequence of Theorem~\ref{t:constr-res} that $\GCharg{\G}$ admits a (projective) symplectic resolution. Therefore, it remains to show that if $(K,g)$ is not listed in (c) then $\GCharg{K}$ does not admit a projective symplectic resolution. 

\begin{defn}
  Let a \emph{symplectic cone} mean a conical symplectic singularity, where the symplectic structure has positive weight under dilations (where the functions on the cone, by definition, have nonnegative weight). We say that a variety is \emph{formally locally symplectically conical} at a point if the formal neighborhood of a point is isomorphic to a formal neighborhood of a symplectic cone.
\end{defn}

It has been conjectured by Kaledin, \cite[Conjecture 1.8]{KaledinSurvey}, that every symplectic singularity is formally locally symplectically conical. Therefore, Proposition~\ref{p:formal-neighborhood} below can be interpreted as saying that Kaledin's conjecture holds for all character varieties of type $A$ and their relative minimal models.  

Recall that $\mathbb{G} = H \times \prod_i SL_{n_i}$ and $Z_0 \subset \mathbb{G}$ a finite central subgroup such that $G = \mathbb{G}/Z_0$. Let $\rho \in \Hom(\pi, \mathbb{G})$ be a semi-simple representation. The stabilizer $\mathbb{G}_{\rho}$ is reductive.  

\begin{lem}\label{lem:formalGGnbhd}
	There exists a homogeneous formal cone $\widehat{Y}_0$ such that:
	\begin{enumerate}
		\item[(i)] $\mathbb{G}_{\rho}$ acts on $\widehat{Y}_0$ preserving the dilation filtration and $\widehat{Y}_0 \git \, \mathbb{G}_{\rho}$ is a formal symplectic cone. 
		\item[(ii)] There is a $\mathbb{G}$-equivariant isomorphism 
		\begin{equation}\label{eq:formalGGnbhd}
					\widehat{\Hom}(\pi,\mathbb{G})_{\mathbb{G} \cdot \rho} \cong \widehat{Y}_{0} \times_{\mathbb{G}_{\rho}} \mathbb{G}
		\end{equation}
			of homogeneous cones inducing an isomorphism of symplectic cones
			$$
			\widehat{\GCharg{\mathbb{G}}}_{[\rho]} \cong \widehat{\Hom}(\pi,\mathbb{G})_{\mathbb{G} \cdot \rho} \git \, \mathbb{G} \cong \left( \widehat{Y}_{0} \times_{\mathbb{G}_{\rho}} \mathbb{G} \right) \git \, \mathbb{G} \cong  \widehat{Y}_0 \git \, \mathbb{G}_{\rho}. 
			$$
			   
\end{enumerate}
\end{lem}

\begin{proof}
	We begin by noting, as in the proof of Theorem~\ref{thm:Blowup22LS}, that the character variety $\GCharg{\GL_n}$ is
	the moduli space of representations of the group algebra $\C[\pi]$
	of the fundamental group $\pi$ of a genus $g$ surface.  Therefore,
	by \cite[Theorem 6.3, Theorem 6.6]{BGV-NHsCY2} and \cite[Corollary~5.21]{Kaplan-Schedler-mpa}, the formal neighborhood
	of $[\rho] \in \GCharg{\GL_n}$ is isomorphic as Poisson varieties to the formal neighborhood of
	$0$ in a Nakajima quiver variety $\Nak{0}{\alpha}{0}$, since the group algebra $\C[\pi]$ is
	a two-dimensional Calabi--Yau algebra. Moreover, the statement of the lemma, but with $\mathbb{G} = GL_n$, is implied by the result \cite[Theorem 5.16]{Kaplan-Schedler-mpa} applied to our particular situation. Diagram~\ref{eq:trivialbundle1} implies that 
	$$
	\widehat{\Hom}(\pi,GL_n)_{GL_n \cdot \rho} = \widehat{\Hom}(\pi,SL_n)_{GL_n \cdot \rho_0} \widehat{\times} \, \widehat{(\C^{2g})}_x, 
	$$
	where $(\rho_0,x) \in \Hom(\pi,SL_n) \times T$ is a point such that $\Z_{n}^{2g} \cdot (\rho_0,x) = \rho$. Restricting to the formal subscheme $ \widehat{\Hom}(\pi,SL_n)_{GL_n \cdot \rho_0} \times \{ x \}$ in $\widehat{\Hom}(\pi,GL_n)_{GL_n \cdot \rho}$ shows that the statement of the lemma also holds for $\mathbb{G} = SL_n$. 
	
	Finally, for $\mathbb{G} = H \times \prod_i SL_{n_i}$, the character variety $\GCharg{\mathbb{G}}$ equals $H^{2g} \times \prod_i \GCharg{SL_{n_i}}$ and hence the statement of lemma follows by taking the product over all $SL_{n_i}$.

	


\end{proof}

\begin{prop}\label{p:formal-neighborhood}
  Both $\GCharg{\G}$ and its minimal model $\widetilde{\GCharg{\G}}$
  are formally locally symplectically conical at every point.
\end{prop}
\begin{proof}
	Let $x \in \GCharg{G} = \GCharg{\mathbb{G}}/Z_0^{2g}$. We wish to describe the formal neighborhood $\widehat{\GCharg{G}}_x$ of $x$ in $\GCharg{G}$. We would like to apply Theorem~\ref{t:conical-quotients} for the action of $Z_0^{2g}$ on $\GCharg{\mathbb{G}}$, but we can't immediately do so since it isn't \emph{a priori} clear that the stabilizer at the point $[\rho]$ preserves the filtration generated by dilations.
	
	Choose $\rho \in \Hom(\pi,\mathbb{G})$, as in Lemma~\ref{lem:formalGGnbhd}, with $Z_0^{2g} \cdot [\rho] = x$. Let $S = \mathrm{Stab}_{Z_0^{2g}} [\rho]$ denote the stabilizer of the point $[\rho]$. Since $Z_0^{2g}$ is a finite group, whose action on $\Hom(\pi,\mathbb{G})$ commutes with the conjugation action of $\mathbb{G}$, we have  
	$$
	\widehat{\GCharg{G}}_x \cong \left(\widehat{\GCharg{\mathbb{G}}}_{[\rho]} \right) \git \, S \cong \left(\widehat{\Hom}(\pi, \mathbb{G})_{\mathbb{G} \cdot \rho} \git \, \mathbb{G} \right) \git \, S \cong \widehat{\Hom}(\pi, \mathbb{G})_{\mathbb{G} \cdot \rho} \git \, (\mathbb{G} \times S),
	$$
	where the second identification is Lemma~\ref{lem:formalGGnbhd}(ii). 
	
	The group $\mathbb{G} \times S$ automatically preserves the adic filtration on $\widehat{\Hom}(\pi,\mathbb{G})_{\mathbb{G} \cdot \rho}$ since ${\mathbb{G} \cdot \rho}$ is stable under $\mathbb{G} \times S$. Therefore, under the isomorphism \eqref{eq:formalGGnbhd}, the group $\mathbb{G} \times S$ acts on $\widehat{Y}_{0} \times_{\mathbb{G}_{\rho}} \mathbb{G}$ preserving the adic filtration. Identify $\widehat{Y}_{0}$ with a closed subscheme of $\widehat{Y}_{0} \times_{\mathbb{G}_{\rho}} \mathbb{G}$ in the obvious way and let $S_0$ denote the stabilizer of $(0,1)$ in $\widehat{Y}_{0} \times_{\mathbb{G}_{\rho}} \mathbb{G}$. Under the identification, $S_0$ acts on $\widehat{Y}_{0}$ preserving the adic filtration. Since the cone $\widehat{Y}_{0}$ is homogeneous, the adic filtration equals the dilation filtration. Thus, $S_0$ acts on $\widehat{Y}_{0}$ preserving the dilation filtration. Theorem~\ref{t:conical-quotients}(a) implies that this action is conjugate to a homogeneous action. Moreover, as noted in Remark~\ref{rem:formalconjPoisson}, we may assume it also preserves the Poisson structure. Therefore, as in Theorem~\ref{t:conical-quotients}(b), we deduce that $\widehat{Y}_0 \git \, S_0$ is a symplectic cone.
	
	Finally, since 
	$$
	\widehat{\Hom}(\pi, \mathbb{G})_{\mathbb{G} \cdot \rho} \git \, (\mathbb{G} \times S) \cong \left( \widehat{Y}_{0} \times_{\mathbb{G}_{\rho}} \mathbb{G} \right) \git \, (\mathbb{G} \times S) \cong \widehat{Y}_0 \git \, S_0,
	$$
	we deduce that $\widehat{\GCharg{G}}_x$ is a symplectic cone.


 
 Next we consider $\widetilde{\GCharg{\G}}$. Note first that an \'etale quotient induces isomorphisms on formal local rings, so we may assume by the final statement of Theorem~\ref{t:rmm} that $G = H \times K$. In fact, we may assume $G=K$ is semisimple. For clarity, we will treat the cases $g = 1$ and $g > 1$ separately. In the case $g = 1$, we write $K = K_1 / C_1$ as in Theorem~\ref{t:rmm}(i). Then $\widetilde{\GChar{K}{1}} = \widetilde{\GChar{K_1}{1}}/C_1^2$ with $\widetilde{\GChar{K_1}{1}}$ a smooth symplectic variety. It follows from Theorem~\ref{t:conical-quotients}(b) that $\widetilde{\GChar{K}{1}}$ is formally locally symplectically conical at every point; see Remark~\ref{rem:formalconjPoisson} for the symplectic part. 
 
 Similarly, if $g \ge 2$ then we write $K = K_0/C_0$ as in Theorem~\ref{t:rmm}(ii). Then $\widetilde{\GCharg{K}} = \widetilde{\GCharg{K_0}}/C_1^{2g}$ and $\widetilde{\GCharg{K_0}}$ is a product of varieties (one for each simple factor of $K_0$), where each factor is either a smooth symplectic variety, or has the form $\GCharg{SL_n}$. It follows from the first part of the proof of the proposition that every point in $\widetilde{\GCharg{K}}$ is formally locally symplectically conical; again, see Remark~\ref{rem:formalconjPoisson} for the symplectic part.

\end{proof}
\begin{lem}\label{lem:nonformalresolveYGg}
  \begin{enumerate}
  \item[(i)] Let $g=1$ and suppose that the minimal model of Theorem \ref{t:rmm} is not smooth.  Then the local cone at every point of that model is $\Q$-factorial and terminal.
  \item[(ii)] Let $g\geq 2$ and suppose that $x \in \widetilde{\GCharg{G}}$ is a point having a preimage
    in  $\widetilde{\GCharg{\mathbb{G}}}$ whose projection to each factor $\widetilde{\GCharg{\SL_n}}$ is either a smooth point or, for $(g,n) \neq 2$,
    a point in the minimal stratum $C_{(1,n)}$ of $\GCharg{\SL_n}$.  Then the local cone at $x$ is $\Q$-factorial and terminal.
  \end{enumerate}
\end{lem}
\begin{proof} Note that the terminal properties are immediate thanks to \cite{NamikawaNote}, since the codimension of the singular locus of a local cone must be at least that in  the ambient space.  So we only have to show that these local cones are $\Q$-factorial.  In the case of (i), this is a finite quotient of a smooth variety, so this is standard; for instance, see \cite[Section 2.4]{BellSchedler2large} and the references therein. For part (ii), we have to handle a finite quotient of $\widetilde{\GCharg{\mathbb{G}}}$, which is a product of smooth varieties and of factors $\GCharg{\SL_n}$. The proof of Theorem \ref{thm:slsqfactor} in Section \ref{sec:localfactchar} shows that the formal neighborhood of the trivial local system in $\GCharg{\SL_n}$ is locally factorial for $g \geq 2$ and $(n,g) \neq (2,2)$; that then also follows for every point of the minimal stratum. Thus our local cone is a finite quotient of a product of locally factorial cones.  Thanks to \cite{OpenFactorial}, the product is locally factorial, and since we take a finite quotient, again by \cite[Section 2.4]{BellSchedler2large} and the references therein, the result is $\Q$-factorial.
\end{proof}   

Lemma~\ref{lem:nonformalresolveYGg} implies that, in the situation of Theorem \ref{t:class-res} where part (c) is not satisfied, then there exists a point $x$ of
$\widetilde{\GCharg{\G}}$ whose formal neighborhood does not admit a
projective symplectic resolution: this follows by \cite[Lemma
6.16]{BS-srqv} (replacing the variety there by an arbitrary symplectic
cone) since the closed point of this formal neighborhood is singular yet $\Q$-factorial terminal. Let $y$ denote the image of $x$ in $\GCharg{\G}$. By Proposition~\ref{p:formal-neighborhood}, the formal neighborhood of $y$ is a conic symplectic singularity. Therefore, we deduce from \cite[Theorem 2.2]{BSnon} (based on \cite{Namikawa}) that $\GCharg{\G}$ does not admit a (projective) symplectic resolution. This completes the proof of Theorem \ref{t:class-res}.

\appendix
\section{Reductive group actions on local cones}
In this appendix, we prove a generalized version of the following
folklore result (cf., e.g, \cite{Hermann-formal-linearization}). We work over the complex
numbers, but one could replace this with any algebraically closed field of characteristic zero (and perhaps the algebraic
closure requirement is unnecessary).
\begin{thm}\label{t:folklore}
  Let $X$ be a smooth scheme over $\C$ and $G$ a reductive group acting on $X$ with fixed point $x \in X$. Then the action
  of $G$ can be linearized in a formal neighborhood of $x$.
\end{thm}
Note that a linearization here means that, for some choice of formal coordinates $\widehat{\mathcal{O}_{X,x}} \cong \C[\![x_1,\ldots,x_n]\!]$, the action of $G$ is via a linear representation $G \to \GL_n$.  

To generalize this, note that a linearization is the same thing as a \emph{homogeneous} action of $G$ on the polynomial ring $\C[x_1,\ldots,x_n]$ whose induced action on the completion $\C[\![x_1,\ldots,x_n]\!]$ coincides with the original one, up to isomorphism. We generalise this to cones:
\begin{defn}
  A cone is an affine variety $X=\Spec A$ whose coordinate ring $A$ is
  nonnegatively graded and connected ($A_0=\C$ and $A_{<0} = 0$).  It
  is called \emph{homogeneous} if $A$ is generated in a single
  positive degree $m > 0$ (without loss of generality $m=1$).
  \end{defn}
  In general, if $X$ is a cone then it is often called \emph{quasi-homogeneous}.
  \begin{example}
    The type $A_n$ Du Val singularity is $\{xy+z^n=0\} \subseteq \C^3$. This is quasi-homogeneous for
    the grading $|x|=|y|=n, |z|=2$, hence a cone. It is homogeneous if and only if $n=2$.
  \end{example}  
Given a cone $X = \Spec A$, there is a canonical action of the
multiplicative group by dilations. On a homogeneous function
$f \in A_n$ it is given by $\lambda \star f := \lambda^{-n} f$, for
$\lambda \in \C^\times$.  Conversely, we can recover the grading on
$\mathcal{O}(X)=A$ from this action. We say therefore that the
grading, and the resulting filtration
$\mathcal{O}(X)_{\leq n} := \bigoplus_{i \leq n} \mathcal{O}(X)_i$,
are the ones induced by dilations on $X$.
\begin{example}
  Let $X$ be any scheme and $x \in X$. Equip the local ring
  $\mathcal{O}_{X,x}$ with the $\mathfrak{m}_{X,x}$-adic filtration
  (for $\mathfrak{m}_{X,x}$ the maximal ideal). The tangent cone
  $C_{X,x} = \Spec \gr \mathcal{O}_{X,x}$ is then homogeneous.
  In general, $X$ is a homogeneous cone based at $x$ if and only if it is
  isomorphic to $C_{X,x}$.
\end{example}
For a general 
cone $X$, the local ring
$\mathcal{O}_{X,x}$, and its completion, are equipped with two in
general distinct filtrations: the $\mathfrak{m}_{X,x}$-adic
filtration, and the one coming from dilations on $X$ (i.e., the descending filtration generated by homogeneous global function on $X$).
\begin{example}
  Let $X$ be a scheme and $x \in X$. If $G$ is any reductive group
  acting on $X$ fixing $x$, then the categorical quotient
  $C_{X,x}/\!/G = \Spec \gr \mathcal{O}_{X,x}^G$ is also a cone (but
  in general
  not homogeneous). In general,
  $C_{X,x}/\!/G$ and $X/\!/G$ need not be formally isomorphic at the
  image $\overline{x}$ of $x$. If they are, then $X/\!/G$ is formally
  locally a cone at $\overline{x}$.
\end{example}
In view of these examples, Theorem \ref{t:folklore} can be restated as the following: if $X$ is \emph{smooth}
at $x$, then there is an isomorphism of formal neighborhoods of $x$ of the two pairs $(X,G)$ and $(C_{X,x},G)$. Thus, $X/\!/G$ is formally locally a cone
at $\overline{x}$.  

Our main result gives a generalization of this to the case where $X$ need not be smooth and $x$ need not be a fixed point:
\begin{thm}\label{t:conical-quotients}
  Let $G$ be a reductive group acting on a scheme $X$ and let $x \in X$ be a point such that
  the formal neighborhood
  $\widehat{X}_x$ is a cone, i.e., isomorphic to $\widehat{Y}_y$ for $Y$ a cone with cone point $y \in Y$.
  Assume moreover that either:
  \begin{enumerate}
  \item[(i)] $Y$ is homogeneous (i.e., $Y \cong C_{X,x}$); or
  \item[(ii)] The induced action of the stabilizer $G_x$ on $\widehat{Y}_y$ preserves the filtration coming from dilations.
\end{enumerate}
Then the following hold:
\begin{enumerate}
\item[(a)] If $G_x$ is reductive, then the action of $G_x$ on $\widehat{Y}_y$ is conjugate to a homogeneous (i.e., dilation-equivariant) action;
\item[(b)] If $X$ is an affine variety and the orbit $G\cdot x$ is closed, then $X/\!/G$ is formally locally a cone at $\overline{x}$.
\end{enumerate}
\end{thm}
Note that, by Matsushima's criterion, the hypotheses of (b) imply those of (a). Also, note that hypothesis (i) implies hypothesis (ii), since in this case the filtration coming from dilations is the $\mathfrak{m}_{X,x}$-adic filtration. It is clear that this hypothesis is satisfied in the situation of Theorem \ref{t:folklore}, so that it is a special case of Theorem \ref{t:conical-quotients}.
\begin{proof}[Proof of Theorem \ref{t:conical-quotients}]
  As observed, it is enough to assume condition (ii). We first prove (a).  Let $H$ be the group of all automorphisms of $\widehat{Y}_y$ (i.e., of $\widehat{\mathcal{O}_{Y,y}}$) preserving the filtration coming from dilations (this is just the entire automorphism group in the homogeneous case).  Let $H_0 < H$ be
  the subgroup of homogeneous automorphisms of $Y$ (i.e., automorphisms commuting with dilations).  We have a splitting of the inclusion, $\pi_0: H \to H_0$, given
  by taking the associated graded automorphism. Let $U$ be the kernel. Observe that $U$ is pro-unipotent.  We have
  a semi-direct product decomposition $H = H_0 \ltimes U$.

  Now, the action of $G_x$ can be expressed as a homomorphism $\rho: G_x \to H_0 \ltimes U$. Statement (a) amounts to saying that this is conjugate to the homomorphism $\rho_0 = \pi_0 \circ \rho: G_x \to H_0$.  Letting $\pi': H_0 \ltimes U \to U$ be the (set-theoretic) projection to the second factor,
  $\rho$ defines a one-cocycle $\rho'=\pi' \circ \rho: G \to U$, with respect to the action given by $\rho_0$.  Then (a) is equivalent to the statement that $\rho'$ is a coboundary. Since $G_x$ is reductive, $H^1(G_x, \C)=0$ as there are no nontrivial homomorphisms from a reductive group to a connected unipotent one. 
  Since $U$ is pro-unipotent, the long-exact sequences on group cohomology imply also that $H^1(G_x, U) = 0$. This proves (a).

  For part (b), we apply Luna's slice theorem. The construction of the formal slice to $G\cdot x$ at $x \in X$ can be constructed in the formal neighborhood
  $\widehat{X}_x$ so as to be homogeneous with respect to dilations (by finding a section of the Zariski tangent space
  $\mathfrak{m}_{X,x}/\mathfrak{m}_{X,x}^2$ in $\widehat{\mathcal{O}_{X,x}}$ which is equivariant for dilations and the action of $G_x$, and similarly for a complement to the Zariski tangent space of $G \cdot x$,
  see \cite[Section III]{Luna}). The formal neighborhood of $X/\!/G$ at $\overline{x}$ is isomorphic to the quotient of this formal slice by $G_x$, so the result follows from (a).
\end{proof}

\begin{remark}\label{rem:formalconjPoisson}
	If $X$ is a Poisson scheme and $G$ preserves the Poisson structure then the formal cone $\widehat{Y}_y$ of Theorem~\ref{t:conical-quotients} inherits a Poisson structure from $X$. By restricting the group $H$ in the proof of the theorem to automorphisms preserving the Poisson structure, we see that the action of $G_x$ on $\widehat{Y}_y$ is conjugate to a homogeneous action preserving the Poisson structure. In this case, Theorem~\ref{t:conical-quotients}(b) says that $X \git \, G$ is formally locally a Poisson cone at $\overline{x}$.    
\end{remark}


\def\cprime{$'$} \def\cprime{$'$} \def\cprime{$'$} \def\cprime{$'$}
\def\cprime{$'$} \def\cprime{$'$} \def\cprime{$'$} \def\cprime{$'$}
\def\cprime{$'$} \def\cprime{$'$} \def\cprime{$'$} \def\cprime{$'$}
\def\cprime{$'$} \def\cprime{$'$}

\end{document}